\documentclass[12pt,amstex, reqno]{amsart}
\usepackage{mathrsfs}

\usepackage{epsfig}
\usepackage{amsmath}
\usepackage{amssymb}
\usepackage{amscd}
\usepackage{graphicx}

\usepackage{epsfig}
\usepackage{float}

\usepackage{calc}         
\usepackage{graphicx}     
\usepackage{ifthen}       
\usepackage{subfigure}    
\usepackage{pst-all}      
\usepackage{pst-poly}     
\usepackage{multido}

\usepackage[colorlinks,citecolor=blue]{hyperref}

\makeatletter
\@namedef{subjclassname@2020}{%
	\textup{2020} Mathematics Subject Classification}

\newtheorem{thm}{Theorem}[section]
\newtheorem{lem}[thm]{Lemma}

\newtheorem{prop}[thm]{Proposition}

\newtheorem{ex}[thm]{Example}

\theoremstyle{definition}
\newtheorem{de}[thm]{Definition}

\theoremstyle{remark}
\newtheorem{rem}[thm]{Remark}

\numberwithin{equation}{section}

\def \a {\alpha }
\def \b {\beta}

\begin{document}
	\title [Sequence entropy for amenable group actions]
	{Sequence entropy for amenable group actions}
	\author[C. Liu and K. Yan]{Chunlin Liu  and Kesong Yan}
	
	\thanks{C. Liu is partially supported by NNSF of China (12090012, 12090010). K. Yan is supported by
		NNSF of China
		(11861010, 12171175), NSF of Guangxi Province (2018GXNSFFA281008) and Project of First Class Disciplines of
		Statistics and Guangxi Key Laboratory of Quantity Economics.}
	
	\address[C. Liu]{CAS Wu Wen-Tsun Key Laboratory of Mathematics, School of Mathematical Sciences,
		University of Science and Technology of China, Hefei, Anhui, 230026,
		P.R. China} \email{lcl666@mail.ustc.edu.cn}

	\address[K. Yan]{School of Mathematics and Quantitative Economics, Guangxi University of Finance and
		Economics, Nanning, Guangxi, 530003, P. R. China}
	
	\address{Guangxi Key Laboratory of Big Data in Finance and
		Economics, Guangxi University of Finance and Economics, Nanning,
		Guangxi, 530003, P. R. China} \email{ksyan@mail.ustc.edu.cn}

	
	\begin{abstract}
		We study the sequence entropy for amenable group
		actions and investigate systematically spectrum and several
		mixing concepts via sequence entropy both in measure-theoretic
		dynamical systems and topological dynamical systems. Moreover, we
		use sequence entropy pairs to characterize
		weakly mixing and null systems in the topological sense.
	\end{abstract}
	
	
	\maketitle
	
	
	\section{Introduction}
	By a $G$-topological dynamical system (or $G$-system for short) $(X,G)$
	we mean a compact metric space $X$ with a (semi)-group $G$ acting
	continuously on $X$. If $G=\mathbb{Z}_+$ (resp. $\mathbb{Z}$) is the
	additive semigroup of non-negative integers (resp. the additive
	group of integers), then the classical $\mathbb{Z}_+$ (resp.
	$\mathbb{Z}$)-action can be induced by a continuous map (resp. a
	homeomorphism) $T: X \rightarrow X$, and we usually denote it by
	$(X, T)$.
	For a $G$-system $(X, G)$ if there exists a $G$-invariant Borel probability measure $\mu$ on $X$, it induces a measure-theoretic dynamical system
	$(X, \mathcal{B}_X,\mu, G)$ (or $G$-MPS for short), where $\mathcal{B}_X$ is the Borel $\sigma$-algebra on $X$. It is well known that, for a $G$-system, there are some groups such that the set of invariant
	probability measures may be empty, whereas, amenability of the
	group ensures  the existence of invariant
	probability measures (see \cite{Gre69}).
	
	\medskip
	
	For $\mathbb{Z}_+$-actions, discrete spectral and mixing properties
	for measure-theoretic dynamical systems have been investigated by
	many authors from kinds of viewpoints (see for example \cite{Fur81,
		HLTXY21, HWY19, LTY15}). An important way to characterize them is by
	the sequence entropy concept. The related research can be traced back to
	Kushnirenko \cite{Ku67}  who introduced the notion of sequence
	entropy along a given infinite sequence of $\mathbb{Z}_+$ for the
	measure-theoretical dynamical system and proved that an invertible
	measure-theoretical dynamical system has discrete spectrum if and
	only if the sequence entropy of the system with respect to any infinite
	sequence of $\mathbb{Z}_+$ is zero (some new proofs related to this
	result can be found in \cite{HMY04,HSY05}). Later, Hulse
	\cite{Hul79} studied the sequence entropy of measure-theoretical
	dynamical systems with quasi-discrete spectrum. In addition, Saleski
	\cite{Sa77} and Hulse \cite{Hul82} obtained some characterizations
	of weakly mixing and strongly mixing measure-theoretical dynamical
	systems by using the sequence entropy. Moreover, Hulse \cite{Hul86}
	gave the characterizations of the compact and weakly mixing
	extensions of measure-theoretical dynamical systems via conditional
	sequence entropy, and some results related to mild mixing and mildly
	mixing extension can be found in \cite{Zha92, Zha93} by Zhang. In
	\cite{CMS09}, Coronel, Maass and Shao studied the relation between
	sequence entropy and the Kronecker and rigid algebras. As an
	application, they characterized compact, rigid and mixing extensions
	via conditional sequence entropy. More recently,  directional sequence entropy  was introduced to study the directional  discrete spectrum by Liu and Xu \cite{LX2023} and directional weak mixing by Liu \cite{L}.
	
	\medskip
	
	From the viewpoint of topological dynamical systems, there are several
	ways to characterize topologically mixing properties (see for example
	\cite{Gla04,HY02, HY04, HY12, SY04}). In 1974, Goodman \cite{Goo74}
	introduced the notion of topological sequence entropy in a way
	analogous to that of Kushnirenko, and studied some properties of
	topologically null systems. More details of topological sequence entropy for $\mathbb{Z}$-actions can be found in \cite{Canovas}. In \cite{Li02}, Li first obtained a
	characterization of topological weak mixing by using topological
	sequence entropy. Moreover, Huang, Shao and Ye \cite{HSY05} further
	investigated systematically several topological mixing concepts via
	topological sequence entropy. Besides, Huang, Li, Shao and Ye
	\cite{HLSY03} localized the notion of sequence entropy by defining
	sequence entropy pairs and proved that a topological dynamical
	system is topologically weakly mixing if and only if any pair  not
	in the diagonal is a sequence entropy pair.
	
	\medskip
	
	As the research progressed, many results in ergodic theory and
	topological dynamical systems were extended to larger groups, such
	as amenable groups. The reader may read a recent book \cite{KL16} to
	learn more details about systems under amenable group actions. In
	1987, Ornstein and Weiss \cite{OW87} developed a method called
	quasi-tiling, which can serve as the substitute of the Rohlin lemma,
	and extended many known results of entropy theory from
	$\mathbb{Z}$-actions to countable infinite amenable group actions.
	Recently, discrete spectrum and mixing concepts for group actions
	were studied in different settings by many authors. For example,
	Huang, Ye and Zhang \cite{HYZ11} studied systematically the local
	entropy theory for actions of a countable discrete amenable group,
	and used it to discuss transitivity, mild mixing, strong mixing,
	regionally proximal relation along sequence and weak mixing of all
	orders. In \cite{YZZ21}, Yu, Zhang and Zhang studied discrete spectrum for
	amenable group actions from the viewpoint of measure complexity. In
	\cite{YLZ21}, Yan, Liu and Zeng investigate systematically several
	mixing concepts for group actions via weak disjointness, return time
	sets and topological complexity functions. In \cite{Fe}, Garc\'ia-Ramos defined weaker forms of topological and measure-theoretical equicontinuity
	for topological dynamical systems, and studied their relationships with sequence entropy
	and systems with discrete spectrum. Kerr and Li \cite{KLt,KLm} established local combinatorial
	and linear-geometric characterizations of sequence entropy in both topological and measure-theoretical sense. The reader also can see a survey \cite[Section 8]{LYY} for more details.
	
	\medskip
	
	It is a natural question whether we can characterize discrete spectrum
	and mixing properties for amenable group actions using sequence
	entropy. Based on this motivation, in this paper, we shall introduce
	the notion of sequence entropy for amenable group actions along some given F\o lner sequence, and use it
	to investigate discrete spectral properties and kinds of mixing
	concepts in both measure-theoretical and topological sense. Let $G$ be a countable discrete amenable group. We prove that a $G$-MPS
	has discrete spectrum if and only if the sequence entropy of the
	system with respect to a fixed F\o lner sequence along any infinite
	subset of $G$ is zero. Moreover, we show that if a $G$-system $(X,
	G)$ is topologically strongly mixing then each non-trivial
	finite open cover of $X$ and any infinite subset $H$ of $G$, the
	topological sequence entropy with respect to a fixed F\o lner
	sequence along some infinite subset of $H$ is positive; topological mild mixing implies that for each non-trivial finite open cover of $X$ and any IP-set $H$ of $G$, the topological sequence entropy with respect to a fixed F\o lner sequence along some infinite subset of $H$ is positive. In addition,
	for an Abelian group action, we obtain that it is weakly mixing if
	and only if for each non-trivial finite open cover, the topological
	sequence entropy with respect to a fixed F\o lner sequence along
	some infinite subset is positive.
	
	\medskip
	
	\medskip
	This paper is organized as follows. In Section 2, we recall some
	basic notions and properties that we use in this paper.  In Section 3, we give characterizations of
	discrete spectrum and weakly mixing measures via sequence entropy.
	In Section 4, we investigate some notions of topological mixing via
	topological sequence entropy. In Section 5, we introduce sequence entropy pairs for $G$-systems and use them to describe null systems
	and weakly mixing systems.
	
	\section{Preliminaries}
	Recall that an infinite countable discrete group $G$ is called {\it
		amenable} if there exists a sequence of finite subsets $F_n \subset
	G$ such that for every $g \in G$,
	\begin{equation} \label{eq:2-1}
		\lim_{n \rightarrow +\infty} \frac{|g F_n \triangle F_n|}{|F_n|}=0,
	\end{equation}
	where $|\cdot|$ denotes the cardinality of a set and $\triangle$
	stands for the symmetric difference of sets. A sequence satisfying
	condition \eqref{eq:2-1} is called a {\it F\o lner sequence} (see
	\cite{Fol55}).
	
	\medskip
	
	Throughout this paper, we let $X$ be a compact metric space with a
	metric $d$ and $G$ an infinite countable discrete group. By
	a {\it $G$-system} we mean  a pair $(X, G)$, where $\Gamma: G
	\times X \rightarrow X, (g,x) \mapsto gx$ is a continuous mapping
	satisfying
	\begin{enumerate}
		\item $\Gamma(e_G,x)=x$ for each $x \in X$;
		
		\medskip
		
		\item $\Gamma(g_1, \Gamma(g_2,x))=\Gamma(g_1g_2,x)$ for each $g_1,g_2 \in
		G$ and $x \in X$.
	\end{enumerate}
	If $X$ is a single point set then $(X, G)$ is said to be a {\it
		trivial system}.
	
	\medskip
	If not explicitly stated, in this paper, we always suppose that $G$ is an infinite countable discrete amenable group and
	$\mathbf{F}=\{F_n\}_{n=1}^{\infty}$ is a F\o lner sequence of $G$
	with $e_G \in F_1 \subseteq F_2 \subseteq \cdots$ and
	$\bigcup_{n=1}^{\infty}F_n=G$.
	
	\medskip
	
	For two dynamical systems $(X, G)$ and $(Y, G)$, their {\it product
		system} $(X \times Y, G)$ is defined by the diagonal action:
	$g(x,y)=(gx,gy)$ for all $x \in X, y \in Y$ and $g \in G$. Higher
	order products are defined analogously.
	
	\medskip
	
	Let $(X, G)$ be a $G$-system,  $\mathcal{B}_X$  be the collection of Borel subsets of $X$ and $\mathcal{M}(X)$ be the set of all
	Borel probability measures on $X$. We say that $\mu \in \mathcal{M}(X)$ is
	{\it $G$-invariant} if $\mu=g\mu:=\mu \circ g^{-1}$ for each $g \in
	G$; a $G$-invariant measure $\nu \in \mathcal{M}(X)$ is called
	{\it ergodic} if $\nu(\bigcup_{g \in G} gA)=0$ or $1$ for any $A\in\mathcal{B}_X$. Denote by $\mathcal{M}(X, G)$ and $\mathcal{M}^e(X, G)$ the
	set of all $G$-invariant Borel probability measures and ergodic
	$G$-invariant Borel probability measures on $X$, respectively. The
	amenability of $G$ ensures that $\mathcal{M}(X, G)\neq \emptyset$.
	Under the weak$^*$-topology, $\mathcal{M}(X, G)$ is a convex compact
	metric space, and $\mu \in \mathcal{M}^e(X, G)$ if and only
	if it is an extreme point of $\mathcal{M}(X, G)$ (see \cite[Theorem
	4.2]{Gla03}).
	
	\medskip
	
	Note that each probability measure $\mu \in \mathcal{M}(X, G)$ induces a
	$G$-MPS $(X,
	\mathcal{B}_X, \mu, G)$. Let $H_c \subset L^2(X, \mathcal{B}_X,
	\mu)$ be the closed subspace generated by all finite dimensional
	$G$-invariant subspaces of $L^2(X, \mathcal{B}_X, \mu)$. It is
	well known that for $f \in L^2(X, \mathcal{B}_X, \mu)$, $f \in H_c$
	if and only if $f$ is an almost periodic function, that is, $\{U_gf:
	g \in G\}$ is precompact in $L^2(X, \mathcal{B}_X, \mu)$, where
	$U_gf:=f \circ g$ for any $g \in G$. By \cite[Theorem 7.1]{Zim76},
	there exists a $G$-invariant sub-$\sigma$-algebra
	$\mathcal{K}_{\mu}$ of $\mathcal{B}_X$ such that $H_c=L^2(X,
	\mathcal{K}_{\mu}, \mu)$. We call $\mathcal{K}_{\mu}$ the {\it
		Kronecker algebra} of $(X, \mathcal{B}_X, \mu, G)$. We say that
	$\mu$ has {\it discrete spectrum} if $H_c=L^2(X, \mathcal{B}_{X},
	\mu)$ or equivalently $\mathcal{K}_{\mu}=\mathcal{B}_X$.
	
	\medskip
	
	We recall that a $G$-MPS $(X, \mathcal{B}_X, \mu, G)$ is {\it weakly
		mixing} if the product system $(X \times X, \mathcal{B}_X \times
	\mathcal{B}_X, \mu \times \mu, G)$ is ergodic. For an amenable group
	action, weak mixing can be expressed in terms of asymptotic conditions
	involving averages along a F\o lner sequence (see \cite[Theorem
	4.21]{KL16}).
	
	\begin{thm} \label{thm:2-1}
		Given a $G$-MPS $(X, \mathcal{B}_X, \mu, G)$, it is weakly mixing if
		and only if
		$$\lim_{n \rightarrow \infty} \frac{1}{|F_n|}\sum_{g \in F_n} |\mu(A \cap g^{-1}B)-\mu(A)\mu(B)|=0$$
		for all $A, B \in \mathcal{B}_X$.
	\end{thm}
	
	\medskip
	
	Given a subset $S$ of $G$, the {\it upper density
	} and {\it lower density} of $S$ with respect to a F\o lner sequence
	$\mathbf{F}=\{F_n\}_{n=1}^\infty$ are defined by
	\begin{equation*}
		\overline{D}_{\mathbf{F}}(S)=\limsup_{n \rightarrow \infty}\frac{|S
			\cap F_n|}{|F_n|}  \mbox{\ \ and\ \ }
		\underline{D}_{\mathbf{F}}(S)=\liminf_{n \rightarrow \infty}\frac{|S
			\cap F_n|}{|F_n|}.
	\end{equation*}
	We say that $A$ has {\it density} $D_{\mathbf{F}}(S)$ with respect
	to $\mathbf{F}$ if
	$D_{\mathbf{F}}(S)=\overline{D}_{\mathbf{F}}(S)=\underline{D}_{\mathbf{F}}(S)$.
	
	\medskip
	
	\begin{thm} \label{thm:2-2}
		A given $G$-MPS $(X, \mathcal{B}_X, \mu, G)$ is weakly mixing if
		and only if for every pair of elements $A, B \in \mathcal{B}_X$,
		there is a subset $S=\{s_i\}_{i=1}^{\infty}$ of $G$ with
		$D_{\mathbf{F}}(S)=1$ such that
		$$\lim_{i \rightarrow \infty} \mu(A \cap s_i^{-1}B)=\mu(A)\mu(B).$$
	\end{thm}
	
	\begin{proof}
		This directly follows the idea of the proof of \cite[Theorem
		1.20]{Wal82}.
	\end{proof}
	
	\medskip
	
	The following decomposition theorem is well-known (see \cite[Theorem
	2.24]{KL16}), which can be viewed as an amenable version of the
	Koopman-Von Neumann spectrum mixing theorem for $\mathbb{Z}$-actions
	(see \cite{Ber96, KN32}).
	
	\begin{thm} \label{thm:2-3}
		Let $(X, \mathcal{B}_X, \mu, G)$ be a
		$G$-MPS. Then the Hilbert space $H=L^2(X, \mathcal{B}_X, \mu)$ can
		be decomposed as
		$$H=L^2(X, \mathcal{K}_{\mu}, \mu) \oplus WM(X, G),$$
		where
		$$WM(X,G)=\left\{f \in H: \lim_{n \rightarrow \infty}\frac{1}{|F_n|}\sum_{g \in F_n}\left|\langle U_g f, h\rangle\right|=0
		\mbox{\ for all\ } h \in H\right\}.$$
	\end{thm}
	
	\begin{rem} \label{rem:2-3}
		It follows  directly from Theorems \ref{thm:2-1} and \ref{thm:2-3}
		that $(X, \mathcal{B}_X, \mu, G)$ is weakly mixing if and only if
		$\mathcal{K}_{\mu}$ is trivial, that is, $\mathcal{K}_{\mu}=\{X,\emptyset\}$.
	\end{rem}
	
	\medskip
	
	Let $(X, G)$ be a $G$-system and $\mu \in \mathcal{M}(X,G)$. Given a
	finite measurable partition $\alpha$ of $X$ and a
	sub-$\sigma$-algebra $\mathcal{A}$ of $\mathcal{B}_X$, we define the conditional entropy of $\alpha$ given $\mathcal{A}$ as
	\begin{equation*}
		H_{\mu}(\alpha|\mathcal{A}):=\sum_{A \in \alpha} \int_X \,
		-\mathbb{E}(1_A|\mathcal{A}) \log \mathbb{E}(1_A|\mathcal{A}) \,
		\mathrm{d}\mu,
	\end{equation*}
	where $\mathbb{E}(1_A|\mathcal{A})$ denotes the conditional
	expectation of $1_A$ with respect to $\mathcal{A}$. It is well known
	that $H_{\mu}(\alpha|\mathcal{A})$ increases with respect to
	$\alpha$ and decreases with respect to $\mathcal{A}$. Set
	$\mathcal{T}=\{\emptyset, X\}$ and define
	\begin{equation*}
		H_{\mu}(\alpha):=H_{\mu}(\alpha|\mathcal{T})=\sum_{A \in \alpha}
		-\mu(A) \log \mu(A).
	\end{equation*}
	It is easy to check that $H_{\mu}(\alpha|\beta)=H_{\mu}(\alpha \vee
	\beta)-H_{\mu}(\beta)$ for any finite measurable partitions $\alpha$
	and $\beta$. More generally, for a sub-$\sigma$-algebra $\mathcal{A}
	\subset \mathcal{B}_X$, we have
	\begin{equation*}
		H_{\mu}(\alpha\vee \beta|\mathcal{A})=H_{\mu}(\beta|\mathcal{A})+
		H_{\mu}(\alpha|\beta \vee \mathcal{A}).
	\end{equation*} A finite measurable partition $\a$ is finer than another finite measurable partition  $\b$ if for any $A\in\a$ , there exists $B\in \b$ such that $A\subset B$, denoted by $\a\succeq \b$. It is well known that if $\a\succeq \b$ then $H_{\mu}(\a)\ge H_{\mu}(\b)$.
	\medskip
	
	Given an infinite subset $S$ of $G$ and a F\o lner sequence $\textbf{F}=\{F_n\}_{n=1}^\infty$, for
	$\mu \in \mathcal{M}(X, G)$ and a finite measurable partition
	$\alpha$ of $X$, we define the {\it sequence entropy of $\alpha$
		with respect to $\mu$ and $\mathbf{F}$ along $S$} by
	$$h_{\mu}^{S, \mathbf{F}}(G, \alpha)=\limsup_{n \rightarrow \infty}\frac{1}{|S \cap F_n|} H_{\mu}
	\left(\bigvee_{g \in S \cap F_n} g^{-1}\alpha\right).$$ The {\it
		sequence entropy of $(X, \mathcal{B}_X, \mu, G)$ with respect to
		$\mathbf{F}$ along $S$} is defined by
	$$h_{\mu}^{S, \mathbf{F}}(G)=\sup_{\alpha} h_{\mu}^{S, \mathbf{F}}(G, \alpha),$$
	where the supremun is taken over all finite measurable partitions of
	$X$.
	
	Following ideas in the proof of similar result for entropy, we can extend it to sequence entropy (see for example \cite[Theorem 4.22]{Wal82}).
	\begin{prop}\label{prop:2-5}
		Let $(X,\mathcal{B}_X,\mu,G)$ be a $G$-MPS.
		Suppose that $\{\alpha_n\}_{n=1}^{\infty}$ is a sequence of finite
		measurable partitions of $X$ with $\alpha_n \nearrow \mathcal{B}_X$.
		Then for any infinite sequence $S$ of $G$, we have
		\begin{equation*}
			\lim_{n \rightarrow
				\infty}h^{S,\mathbf{F}}_{\mu}(G,\alpha_n)=h^{S,\mathbf{F}}_{\mu}(G).
		\end{equation*}
	\end{prop}
	
	Let  $(X, G)$ be a $G$-system and $S$ an infinite subset of $G$.
	For a
	finite open cover $\mathcal{U}$ of $X$, let $\mathcal{N}(\mathcal{U})$ denote
	the minimal cardinality among all cardinalities of subcovers of
	$\mathcal{U}$. Then we define the {\it topological sequence entropy
		of $\mathcal{U}$ with respect to $\mathbf{F}$ along $S$} by
	$$h_{\mathrm{top}}^{S, \mathbf{F}}(G, \mathcal{U})=\limsup_{n \rightarrow \infty}\frac{1}{|S \cap F_n|}
	\log \mathcal{N}\left(\bigvee_{g \in S \cap F_n} g^{-1}\mathcal{U}\right).$$
	The {\it topological sequence entropy of $(X, G)$ with respect to
		$\mathbf{F}$ along $S$} is defined by
	$$h_{\mathrm{top}}^{S, \mathbf{F}}(G)=\sup_{\mathcal{U}} h_{\mu}^{S, \mathbf{F}}(G, \mathcal{U}),$$
	where the supremun is taken over all finite open covers of $X$.
	
	\medskip
	\begin{rem}
		In fact, the sequence entropy can be defined by just picking any F\o lner sequence of the group $G$.
		The reason why we fix a F\o lner sequence to define the sequence is stated as follows:
		\begin{enumerate}
			\item For $\mathbb{Z}$-actions, the topological sequence entropy and metric entropy are
			defined by fixing the F\o lner sequence $\{[0,n]\}_{n=1}^\infty$ and so
			for $G$-actions, we also define them similarly.
			\item The notions of mixing and spectrum can be defined by fixing a F\o lner sequence,
			and, in fact, the definitions are independent of  the choice of F\o lner sequences.
			In this paper, we will establish the relation between these notions and sequence entropy.
			Therefore, to some extent, sequence entropy is independent of the choice of F\o lner sequences.
			Then, when we study some properties,
			we may just fix a single F\o lner sequence to study them instead of considering all F\o lner sequences of $G$.
		\end{enumerate}
	\end{rem}
	\section{Discrete spectrum measures and sequence entropy}
	In this section, we give characterizations of discrete
	spectrum and weakly mixing measures via sequence entropy, which follow the ideas in
	\cite{HMY04} and \cite{HSY05}.
	
	Firstly, we discuss properties of
	the Kronecker algebra $\mathcal{K}_{\mu}$. Let us begin with the following
	lemmas.
	
	\medskip
	
	\begin{lem} \label{lem:3-1}
		{\rm (\cite[Lemma 4.15]{Wal82})} Let $(X,\mathcal{B}_X,\mu)$ be
		a Borel probability space and $r\geq 1$ be a fixed integer. For each
		$\epsilon >0$, there exists $\delta>0$ such that if
		$\alpha=\{A_1,A_2,\ldots,A_r\}$ and $\beta=\{B_1,B_2,\ldots,B_r\}$
		are any two finite measurable partitions of $(X,\mathcal{B}_X,\mu)$
		with $\sum_{j=1}^r\mu(A_j \bigtriangleup B_j)< \delta$, then
		$H_\mu(\alpha|\beta)+H_\mu(\beta|\alpha)<\epsilon$.
	\end{lem}
	
	\begin{lem} \label{lem:3-2}
		Let $(X,\mathcal{B}_X,\mu,G)$ be a $G$-MPS. Then for any finite
		measurable partition $\alpha \subset \mathcal{K}_{\mu}$ and any
		infinite subset $S$ of $G$, we have $h_{\mu}^{S, \mathbf{F}}(G,
		\alpha)=0$.
	\end{lem}
	
	\begin{proof}
		Let $\alpha=\{A_1, A_2, \cdots, A_n\}$ with $A_i \in
		\mathcal{K}_{\mu}$ for all $i=1, 2, \ldots, n$. Since
		$\bigvee_{i=1}^n\{A_i, X \setminus A_i\} \succeq \alpha$, to show
		that $h_{\mu}^{S, \mathbf{F}}(G, \alpha)=0$ for any infinite subset
		$S$ of $G$, it suffices to show that $h_{\mu}^{S, \mathbf{F}}(G,
		\{B, X \setminus B\})=0$ for any $B \in \mathcal{K}_{\mu}$ and any
		infinite subset $S$ of $G$.
		
		\medskip
		
		Let $B \in \mathcal{K}_{\mu}$ and $\eta=\{B, X \setminus B\}$. Since
		$\{U_g1_B: g\in G\}$ is precompact in $L^2(X,\mathcal{B}_X,\mu)$,
		for any $\delta >0$, there exists $N \in \mathbb{N}$ such that for
		any $g\in S$, $$\mu\left(g^{-1}B \bigtriangleup
		h_g^{-1}B\right)=\|U_g1_B-U_{h_g}1_B\|<\delta$$ for some $h_g\in F_N
		\cap S$. Thus, by Lemma \ref{lem:3-1}, one has that for any
		$\epsilon>0$, there exists $N \in \mathbb{N}$ such that for any
		$g\in S$,
		$$H_\mu\left(g^{-1}\eta| h_g^{-1}\eta\right)+H_\mu\left(
		h_g^{-1}\eta|g^{-1}\eta\right)<\epsilon$$ for some $h_g \in F_N \cap
		S$. This implies that for any $g \in S$,
		$$H_\mu\left(g^{-1}\eta|\bigvee_{h \in F_N \cap S} h^{-1}\eta\right)\leq
		H_\mu\left(g^{-1}\eta| h_g^{-1}\eta\right)<\epsilon.$$ Furthermore,
		we have
		\begin{eqnarray*}
			&& h_{\mu}^{S,\mathbf{F}}(G,\eta)=\limsup_{n \rightarrow
				\infty}\frac{1}{|F_n\cap S|}H_\mu\left(\bigvee_{g\in F_n\cap
				S}g^{-1}\eta\right)\\\\
			&=& \limsup_{n \rightarrow \infty}\frac{1}{|F_n\cap
				S|}H_\mu\left(\bigvee_{g\in F_N\cap S}g^{-1}\eta \vee \bigvee_{g \in
				(F_n \setminus F_N) \cap S} g^{-1}\eta\right)\\\\
			&=& \limsup_{n \rightarrow \infty}\frac{1}{|F_n\cap
				S|}\left[H_\mu\left(\bigvee_{g\in F_N\cap S}g^{-1}\eta\right) +
			H_{\mu}\left(\bigvee_{g \in (F_n \setminus F_N) \cap S} g^{-1}\eta|
			\bigvee_{g \in F_N \cap S}g^{-1}\eta\right)\right]\\\\
			&\leq& \limsup_{n \rightarrow \infty}\left[\frac{|F_N|}{|F_n\cap
				S|}H_{\mu}(\eta)+\frac{1}{|F_n\cap S|} \sum_{h \in (F_n \setminus
				F_N) \cap S} H_{\mu}\left(h^{-1} \eta |\bigvee_{g \in F_N \cap
				S}g^{-1}\eta\right)\right]\\\\
			&\leq& \epsilon.
		\end{eqnarray*}
		Since $\epsilon>0$ is arbitrary, one has
		$h^{S,\textbf{F}}_{\mu}(T,\eta)=0$. This finishes the proof.
	\end{proof}
	
	\begin{lem} \label{lem:3-3}
		Let $(X,\mathcal{B}_X,\mu,G)$ be a $G$-MPS. Then for any finite
		measurable partition $\alpha$ of $X$ and any infinite subset $S$ of
		$G$, we have $h^{S,\mathbf{F}}_{\mu}(G,\alpha)\leq
		H_\mu(\alpha|\mathcal{K}_\mu)$.
	\end{lem}
	
	\begin{proof}
		Since $(X, \mathcal{B}_X)$ is separable, there exist countably many
		finite measurable partitions $\{\eta_k: k \in \mathbb{N}\}\subset
		\mathcal{K}_\mu$ such that $$\lim\limits_{k \rightarrow \infty
		}H_{\mu}(\alpha|\eta_k)=H_\mu(\alpha|\mathcal{K}_\mu).$$ For a fixed
		$k\in\mathbb{N}$ and an infinite subset $S$ of $G$, by Lemma
		\ref{lem:3-2}, one has
		\begin{equation}\label{3-1}
			\lim_{n\to \infty}\frac{1}{|F_n\cap S|}H_\mu\left(\bigvee_{g\in
				F_n\cap S}g^{-1}\eta_k\right)=0.
		\end{equation}
		Therefore, we have
		\begin{align*}
			h^{S,\textbf{F}}_{\mu}\left(G,\alpha\right)
			&=\limsup_{n\to \infty}\frac{1}{|F_n\cap S|}H_\mu\left(\bigvee_{g\in F_n\cap S}g^{-1}\alpha\right)\\
			&\leq \limsup_{n\to \infty}\frac{1}{|F_n\cap S|}H_\mu\left(\bigvee_{g\in F_n\cap S}g^{-1}\left(\alpha \vee \eta_k\right)\right)\\
			&=\limsup_{n\to \infty}\frac{1}{|F_n\cap
				S|}\left[H_\mu\left(\bigvee_{g\in F_n\cap S}g^{-1}\eta_k\right)+
			H_\mu\left(\bigvee_{g\in F_n\cap S}g^{-1}\alpha|\bigvee_{g\in F_n\cap S}g^{-1}\eta_k\right)\right]\\
			&\overset{\eqref{3-1}}=\limsup_{n\to \infty}\frac{1}{|F_n\cap S|}H_\mu\left(\bigvee_{g\in F_n\cap S}g^{-1}\alpha|\bigvee_{g\in F_n\cap S}g^{-1}\eta_k\right)\\
			&\leq \limsup_{n\to \infty}{\frac{1}{|F_n\cap S|}\sum_{g\in F_n\cap S}H_\mu\left(g^{-1}\alpha|g^{-1}\eta_k\right)}\\
			&=H_\mu\left(\alpha|\eta_k\right).
		\end{align*}
		By letting $k \rightarrow \infty$, one has
		$h^{S,\textbf{F}}_{\mu}(G,\alpha)\leq
		H_\mu(\alpha|\mathcal{K}_\mu)$, which completes the proof.
	\end{proof}
	
	Furthermore, we have the following result.
	
	\begin{thm}\label{thm:3-4}
		Let $(X,\mathcal{B}_X,\mu,G)$ be a $G$-MPS. Then for any finite
		measurable partition $\alpha$ of $X$, there is an infinite subset
		$S$ of $G$ such that $h^{S,\mathbf{F}}_{\mu}(G,\alpha)=
		H_\mu(\alpha|\mathcal{K}_\mu)$.
	\end{thm}
	
	\begin{proof}
		Note that for any $A \in \mathcal{B}_X$,
		$1_A-\mathbb{E}(1_A|\mathcal{K}_{\mu}) \in WM(X,G)$. Thus, by Theorem
		\ref{thm:2-2}, there exists a subset $S'=\{s_i'\}_{i=1}^{\infty}$ of
		$G$ with $D_{\mathbf{F}}(S)=1$ such that
		$$\lim_{i \rightarrow \infty} \langle U_{s_i'}(1_A-\mathbb{E}(1_A|\mathcal{K}_{\mu})), 1_B\rangle=0$$
		for all $B \in \mathcal{B}_X$.
		
		\medskip
		
		Given a finite measurable partition $\beta$ of $X$, let
		$\alpha=\{A_1, A_2, \ldots, A_l\}$ and $\beta=\{B_1, B_2, \ldots,
		B_t\}$. Then, by the discussion above-mentioned, there exists a subset
		$S''=\{s_i''\}_{i=1}^\infty$ of $G$ with $D_{\mathbf{F}}(S'')=1$
		such that
		\begin{equation}\label{eq:3-2}
			\lim_{i\rightarrow \infty}\langle
			U_{s_i''}(1_{A_k}-\mathbb{E}(1_{A_k}|\mathcal{K}_{\mu})),1_{B_j}\rangle=0
		\end{equation}
		for any $1\leq k \leq l$ and $1\leq j\leq t$. Hence
		\begin{align*}
			&\liminf_{i\rightarrow \infty}{H_\mu(s_i''^{-1}\alpha|\beta)}\\
			=&\liminf_{i\rightarrow \infty}{\sum_{k,j}-\mu\left(s_i''^{-1}A_k\cap B_j\right)}
			\log{\left(\frac{\mu(s_i''^{-1}A_k\cap B_j)}{\mu(B_j)}\right)}\\
			=&\liminf_{i\rightarrow \infty}{\sum_{k,j}-\langle U_{s_i''}1_{A_k},1_{B_j}\rangle\log{\left(\frac{\langle U_{s_i''}1_{A_k},1_{B_j}\rangle}{\mu\left(B_j\right)}\right)}}\\
			\overset{\eqref{eq:3-2}}=&\liminf_{i\rightarrow\infty}{\sum_{k,j}-\langle
				U_{s_i''}\mathbb{E}\left(1_{A_k}|\mathcal{K}_{\mu}\right),1_{B_j}\rangle\log{\left(\frac{\langle
						U_{s_i''}\mathbb{E}\left(1_{A_k}|\mathcal{K}_{\mu}\right),1_{B_j}\rangle}{\mu\left(B_j\right)}\right)}}.
		\end{align*}
		Let $$a_{kj}^i=-\langle
		U_{s_i''}\mathbb{E}\left(1_{A_k}|\mathcal{K}_{\mu}\right),1_{B_j}\rangle
		\log{\left(\frac{\langle
				U_{s_i''}\mathbb{E}\left(1_{A_k}|\mathcal{K}_{\mu}\right),1_{B_j}\rangle}{\mu\left(B_j\right)}\right)}$$
		and $$\mu_{B_j}(\cdot)=\frac{\mu(\cdot\cap B_j)}{\mu(B_j)}.$$ By the
		concavity of $-x\log{x}$, we conclude that
		\begin{align*}
			\frac{a_{kj}^i}{\mu\left(B_j\right)}
			&=-\left(\int_{B_j}{\frac{U_{s_i''}\mathbb{E}\left(1_{A_k}|\mathcal{K}_{\mu}\right)}
				{\mu\left(B_j\right)}}\,\mathrm{d}\mu\right)\log{\left(\int_{B_j}
				{\frac{U_{s_i''}\mathbb{E}\left(1_{A_k}|\mathcal{K}_{\mu}\right)}{\mu\left(B_j\right)}}\,\mathrm{d}\mu\right)}\\
			&=-\left(\int_{B_j}{{U_{s_i''}\mathbb{E}\left(1_{A_k}|\mathcal{K}_{\mu}\right)}}\,\mathrm{d}\mu_{B_j}\right)
			\log{\left(\int_{B_j}{{U_{s_i''}\mathbb{E}\left(1_{A_k}|\mathcal{K}_{\mu}\right)}}\,\mathrm{d}\mu_{B_j}\right)}\\
			&\geq -\int_{B_j}{{U_{s_i''}\mathbb{E}\left(1_{A_k}|\mathcal{K}_{\mu}\right)}}
			\log{\left(U_{s_i''}\mathbb{E}\left(1_{A_k}|\mathcal{K}_{\mu}\right)\right)}\,\mathrm{d}\mu_{B_j}\\
			&=-\int_{B_j}{\frac{U_{s_i''}\mathbb{E}\left(1_{A_k}|\mathcal{K}_{\mu}\right)}{\mu\left(B_j\right)}}
			\log{\left(U_{s_i''}\mathbb{E}\left(1_{A_k}|\mathcal{K}_{\mu}\right)\right)}\,\mathrm{d}\mu.
		\end{align*}
		Therefore, we have
		\begin{align*}
			\sum_{k,j}a_{kj}^i
			&\geq- \sum_{k,j}\int_{B_j}{U_{s_i''}\mathbb{E}\left(1_{A_k}|\mathcal{K}_{\mu}\right)}
			\log{\left(U_{s_i''}\mathbb{E}\left(1_{A_k}|\mathcal{K}_{\mu}(G)\right)\right)}\,\mathrm{d}\mu\\
			&=-\sum_k \int_X{U_{s_i''}\mathbb{E}\left(1_{A_k}|\mathcal{K}_{\mu}\right)}
			\log{\left(U_{s_i''}\mathbb{E}\left(1_{A_k}|\mathcal{K}_{\mu}\right)\right)}\,\mathrm{d}\mu\\
			&=-\sum_k \int_X{\mathbb{E}\left(1_{A_k}|\mathcal{K}_{\mu}\right)}
			\log{\left(\mathbb{E}\left(1_{A_k}|\mathcal{K}_{\mu}\right)\right)}\,\mathrm{d}\mu\\
			&=H_{\mu}(\alpha|\mathcal{K}_{\mu}).
		\end{align*}
		This shows that
		\begin{equation}\label{eq:3-3}
			\liminf_{i\to \infty}{H_\mu\left(s_i''^{-1}\alpha|\beta\right)} \geq
			H_\mu\left(\alpha|\mathcal{K}_\mu\right).
		\end{equation}
		
		By using \eqref{eq:3-3} repeatedly, we can obtain inductively an infinite
		subset $S=\{s_i\}_{i=1}^{\infty}$ of $G$ such that for each $i \in
		\mathbb{N}$, one has
		$$H_\mu\left(s^{-1}_i\alpha|\bigvee_{j=1}^{i-1} s^{-1}_j\alpha\right)\geq H_\mu(\alpha|\mathcal{K}_\mu)-\frac{1}{2^i}.$$
		For each $n \in \mathbb{N}$, there exists $1 \leq
		i_1<i_2<\ldots<i_{k_n}$ such that $$F_n \cap S=\{s_{i_1}, s_{i_2},
		\ldots, s_{i_{k_n}}\}.$$ Hence
		\begin{eqnarray*}
			H_{\mu}\left(\bigvee_{g \in F_n \cap S}g^{-1}\alpha\right) &=&
			H_{\mu}\left(\bigvee_{j=1}^{k_n}s_{i_j}^{-1}\alpha\right)\\
			&=&
			H_{\mu}(s_{i_1}^{-1}\alpha)+H_{\mu}(s_{i_2}^{-1}\alpha|s_{i_1}^{-1}\alpha)+\ldots+
			H_{\mu}\left(s_{i_{k_n}}^{-1}\alpha|\bigvee_{j=1}^{k_n-1}s_{i_j}^{-1}\alpha\right)\\
			&\geq& \sum_{r=1}^{k_n}
			H_{\mu}\left(s_{i_r}^{-1}\alpha|\bigvee_{j=1}^{i_r-1}
			s_j^{-1}\alpha\right)\\
			&\geq&
			\sum_{r=1}^{k_n}\left(H_{\mu}(\alpha|\mathcal{K}_{\mu})-\frac{1}{2^{i_r}}\right)
			\geq |F_n \cap S| \cdot H_{\mu}(\alpha|\mathcal{K}_{\mu})-1.
		\end{eqnarray*}
		Therefore, we have
		\begin{align*}
			h^{S,\mathbf{F}}_{\mu}(G,\alpha) &=\limsup_{n \rightarrow \infty}\frac{1}{|F_n\cap S|}H_{\mu}\left(\bigvee_{g\in F_n\cap S}g^{-1}\alpha\right)\\
			&\geq \limsup_{n\rightarrow\infty}\frac{|F_n \cap S| \cdot
				H_{\mu}(\alpha|\mathcal{K}_{\mu})-1}{|F_n \cap
				S|}=H_{\mu}(\alpha|\mathcal{K}_\mu).
		\end{align*}
		This finishes the proof of Theorem \ref{thm:3-4}.
	\end{proof}
	
	\medskip
	
	We say
	that the $G$-MPS $(X, \mathcal{B}_X, \mu, G)$ is {\it null with
		respect to a F\o lner sequence $\mathbf{F}$}, if $h_{\mu}^{S,
		\mathbf{F}}(G)=0$ for any infinite subset $S$ of $G$. By Lemma
	\ref{lem:3-3} and Theorem \ref{thm:3-4}, we can obtain the following characterization of discrete spectrum via sequence entropy.
	
	\begin{thm} \label{thm:3-5}
		Let $(X,\mathcal{B}_X,\mu,G)$ be a $G$-MPS. Then the following two
		conditions are equivalent:
		\begin{itemize}
			\item[(1)] $\mu$ has discrete spectrum;
			
			\medskip
			
			\item[(2)] $(X,\mathcal{B}_X,\mu,G)$ is null with respect to any F\o lner sequence $\mathbf{F}$;
			\medskip
			
			\item[(3)] $(X,\mathcal{B}_X,\mu,G)$ is null with respect to some F\o lner sequence $\mathbf{F}$.
		\end{itemize}
		In particular, whether or not a $G$-MPS is null is independent of the choice of F\o lner sequences.
	\end{thm}
	
	\begin{proof}
		(1) $\Rightarrow$ (2). Assmue that $\mu$ has discrete spectrum, that is,
		$\mathcal{K}_{\mu}=\mathcal{B}_X$. Then for any F\o lner sequence \textbf{F}, by Lemma \ref{lem:3-3},
		$h^{S,\mathbf{F}}_{\mu}(G,\alpha)\leq H_\mu(\alpha|\mathcal{B}_X)=0$
		for any finite measurable partition $\alpha$ and any infinite subset $S$
		of $G$, which implies $(X,\mathcal{B}_X,\mu,G)$ is null with respect
		to \textbf{F}.
		
		\medskip
		
		(2) $\Rightarrow$ (3). This is trival.
		
		\medskip
		
		(3) $\Rightarrow$ (1). Assume that $(X,\mathcal{B}_X,\mu,G)$ is null
		with respect to some F\o lner sequence $\mathbf{F}$. Given $B \in
		\mathcal{B}_X$, let $\alpha=\{B, X \setminus B\}$. Then by
		Theorem \ref{thm:3-4}, there is an infinite subset $S$ of $G$ such
		that
		$$H_\mu(\alpha|\mathcal{K}_\mu)=h^{S,\mathbf{F}}_{\mu}(G,\alpha)=0,$$
		which implies $B \in K_{\mu}$. Therefore,
		$\mathcal{K}_{\mu}=\mathcal{B}_X$, that is, $\mu$ has discrete
		spectrum.
	\end{proof}
	\begin{rem}
		Given a $G$-MPS $(X,\mathcal{B}_X,\mu,G)$, if its Koopman representation is compact (see Definition 2.22 in \cite{KL16}), then $\mu$ has discrete spectrum and hence its sequence entropy is zero.
	\end{rem}
	
	\medskip
	
	By a proof similar to that of Theorem \ref{thm:3-5}, we can obtain a
	characterization of weak mixing via sequence entropy.
	
	\begin{thm} \label{thm:3-6}
		Let $(X,\mathcal{B}_X,\mu,G)$ be a $G$-MPS. Then the following three conditions are
		equivalent:
		\begin{itemize}
			\item[(1)] $(X,\mathcal{B}_X,\mu,G)$ is weakly mixing;
			
			\item[(2)] for any finite measurable partition $\alpha$ of $X$, there exists some infinite subset $S$ of $G$
			such that $h_{\mu}^{S, \mathbf{F}}(G, \alpha)=H_{\mu}(\alpha)$;

			\item[(3)] for any non-trivial finite measurable partition $\alpha$ of $X$, there exists some infinite subset $S$ of $G$
			such that $h_{\mu}^{S, \mathbf{F}}(G, \alpha)>0$.
		\end{itemize}
	\end{thm}
	
	\begin{rem}
		As an example of $G$-MPS with positive sequence entropy, we consider the Bernoulli actions. Let $Y$ be a Polish space. We consider the product topological space $Y^G$, which is also Polish.The product topology on $Y^G$ is generated by the cylinder sets $\prod_{s \in G} A_s$, where each $A_s$ is open and $A_s=Y$ for all $s$ outside of a finite subset of $G$. This generates the Borel $\sigma$-algebra on $Y^G$.
		Let $\nu$ be a Borel probability measure on $Y$. One can show that there is a unique Borel probability measure $\nu^G$ on $Y^G$ by Kolmogorov's extension theorem. Now we define the action $G$ on $Y^G$ by $(s x)_t=x_{s^{-1} t}$ for all $s, t \in G$ and $x \in Y^G$. This action preserves the measure $\nu^G$, and it is called a Bernoulli action.  Then $(Y^G,\mathcal{B}_{Y^G},\nu^G,G)$ is weak mixing (see for example \cite[Page 38]{KL16}). Thus, by Theorem \ref{thm:3-6}, this system has positive sequence entropy.
	\end{rem}

	\section{Topological mixing and topological sequence entropy}
	
	In this section, we characterize  topological weak
	mixing, and provide some necessary conditions of strong mixing and mild mixing for amenable group actions via topological sequence entropy.
	
	\subsection{Weak mixing} We say that a
	$G$-system $(X, G)$ is {\it (topologically) transitive} if for every
	two nonempty open subsets $U$ and $V$ of $X$, $$N(U, V)=\{g \in G:
	U \cap g^{-1}V \neq \emptyset\}$$ is nonempty. A $G$-system $(X,
	G)$ is called {\it (topologically) weakly mixing} if the product
	system $(X \times X, G)$ is transitive, i.e., for any four nonempty
	open sets $U_1, U_2, V_1, V_2$ of $X$,
	$$N(U_1 \times U_2, V_1 \times V_2)=N(U_1, V_1) \cap N(U_2, V_2) \neq \emptyset.$$
	Or explicitly, there exists a $g \in G$ with $U_1 \cap g^{-1}V_1
	\neq \emptyset$ and $U_2 \cap g^{-1}V_2 \neq \emptyset$. Clearly
	every weakly mixing system is transitive.
	
	\medskip
	
	In order to characterize the weak mixing via sequence entropy, we need the following result.
	\begin{thm} \label{thm:4-1}
		{\rm (\cite[Theorem 1.11]{Gla03} or \cite{Pet70})} For an Abelian group
		$G$ and a $G$-system $(X, G)$, the following conditions are
		equivalent:
		
		\begin{enumerate}
			\item[\textrm{(1)}] $(X, G)$ is weakly mixing;
			
			\medskip
			
			\item[\textrm{(2)}] for any nonempty open sets $U,V$, $N(U, V)$ is nonempty and for every
			four nonempty open sets  $U_1, U_2, V_1, V_2 \subset X$, there exist
			nonempty open sets $U, V$ with $N(U, V) \subset N(U_1, V_1) \cap
			N(U_2, V_2)$;
			
			\medskip
			
			\item[\textrm{(3)}] $N(U, U) \cap N(U, V) \neq \emptyset$ for every nonempty open sets $U, V$ of
			$X$;
		\end{enumerate}
	\end{thm}
	
	Recall that an open cover $\mathcal{U}=\{U_1, U_2, \ldots, U_n\}$ of
	$X$ is called {\it non-trivial} if $U_i$ is not dense in $X$ for
	every $1 \leq i \leq n$, and {\it standard} if $n=2$. An open cover
	$\mathcal{U}=\{U_i\}_{i \in I}$ is called {\it admissible} if $U_i
	\setminus \left(\bigcup_{j \in I, j \neq i}U_j\right)$ has nonempty
	interior for each $i \in I$.
	
	\medskip
	
	Now we provide a characterization of
	weakly mixing systems via sequence entropy.
	
	\begin{thm} \label{thm:4-2}
		Let $G$ be an Abelian group. Then for every $G$-system $(X,G)$, the following statements
		are equivalent:
		
		\begin{enumerate}
			\item $(X, G)$ is weakly mixing;
			
			\medskip
			
			\item for each admissible open cover $\mathcal{U}$, there exists an
			infinite subset $S \subset G$ such that $h_{\mathrm{top}}^{S,
				\mathbf{F}}(G, \mathcal{U})=\log \mathcal{N}(\mathcal{U})$;
			
			\medskip
			
			\item for each non-trivial finite open cover $\mathcal{U}$, there exists an infinite subset
			$S \subset G$ such that $h_{\mathrm{top}}^{S, \mathbf{F}}(G,
			\mathcal{U})>0$;
			
			\medskip
			
			\item for each standard open cover $\mathcal{U}$, there exists an infinite subset $S \subset G$
			such that $h_{\mathrm{top}}^{S, \mathbf{F}}(G, \mathcal{U})>0$.
		\end{enumerate}
	\end{thm}
	
	\begin{proof}
		(2) $\Rightarrow$ (4) and (3) $\Rightarrow$ (4) are trivial.
		
		\medskip
		
		(4) $\Rightarrow$ (1). By contradiction, we assume that $(X, G)$ is not weakly mixing.
		Then by Theorem \ref{thm:4-1} there exist nonempty open sets $U_1$
		and $U_2$ such that
		\begin{equation}\label{eq1}
			N(U_1, U_1) \cap N(U_1, U_2)=\emptyset.
		\end{equation}
		Clearly, $U_1 \cap U_2=\emptyset$. Since $X$ is compact, for each $i=1,2$, we take a closed subset $V_i \subset U_i$ with nonempty interior. Then
		$\mathcal{V}=\{X \setminus V_1, X \setminus V_2\}$ is a standard
		open cover of $X$. By \eqref{eq1}, one has that for any $g \in G$ we have $U_1 \cap
		g^{-1}U_1=\emptyset$ or $U_1 \cap g^{-1}U_2=\emptyset$. Thus, for any $g\in G$, there exists
		$W_g=X \setminus V_1$ or $W_g=X \setminus V_2$ such that $V_1
		\subset g^{-1}W_g$.
		
		\medskip
		
		For any infinite set $S \subset G$ and $n \in \mathbb{N}$, set $S
		\cap F_n=\{g_1, g_2, \ldots, g_{k_n}\}$. If a point $x$ does not belong to
		$\widetilde{V}_0=\bigcap_{i=1}^{k_n}g_i^{-1}(X\setminus V_1)$, then there exists
		$i\in\{1,2,\ldots,k_n\}$ such that
		$g_ix \in V_1$, which implies that for any $g \in G$ we have $g_ix \in g^{-1}W_g$. In particular,
		letting $g=g_ig_j^{-1}$ (for
		$j=1,2,\ldots,k_n$) we have
		\[g_ix\in g_ig_j^{-1}W_{g_jg_i^{-1}}\]
		which implies that
		$$x \in g_j^{-1}W_{g_ig_j^{-1}}.$$
		Thus,
		$$x \in \widetilde{V}_i:=\bigcap_{j=1}^{k_n}g_j^{-1}W_{g_jg_i^{-1}}.$$
		This set depends only on $i$, so we have a subcover
		consisting of $k_n+1$ sets $\widetilde{V}_i$, $i=0, 1, \ldots,
		k_n$.
		
		\medskip
		
		Therefore, for all $n \in \mathbb{N}$, we have
		$$\mathcal{N}\left(\bigvee_{g \in S \cap F_n}g^{-1}\mathcal{V}\right) \leqslant |S \cap F_n|+1.$$
		This implies that $h_{\mathrm{top}}^{S, \mathbf{F}}(G,
		\mathcal{V})=0$.
		
		\medskip
		
		(1) $\Rightarrow$ (2). Assume that $(X, G)$ is weakly mixing
		and $\mathcal{U}=\{U_1, U_2, \ldots, U_l\}$ is an admissible open
		cover. Let $$W_i=\mathrm{int}\left(U_i \setminus \bigcup_{j \neq
			i}U_j\right) \mbox{\ \ for each\ } i=1, 2, \ldots, l.$$ Then $W_1,
		W_2, \ldots, W_l$ are pairwise disjoint nonempty open sets of $X$.
		
		\medskip
		
		{\it Claim}. There exists a sequence $S=\{g_n\}_{n=1}^\infty$ of distinct elements of $G$  and an increasing
		sequence $\{m_n\}_{n=1}^{\infty}$ of positive integers such that, for each $n \geq 1$,
		$g_n\in F_{m_{n}+1}\setminus F_{m_n}$ and, for any $s \in \{1, 2, \ldots, l\}^{n}$,
		$$\bigcap_{i=1}^ng_i^{-1}W_{s(i)} \neq \emptyset.$$
		
		\begin{proof}[Proof of Claim]
			We use induction on $n$. It is obvious that
			the claim holds for $n=1$. Assume that the claim holds for $n=k$,
			next we want to show that the claim also holds for $n=k+1$. By the
			assumption, there exist distinct elements $g_1, g_2, \ldots,
			g_k$ of $G$ and positive integers $m_1<m_2<\ldots<m_k$  such that for each $i\in\{1,2,\ldots,k\}$,
			$g_i\in F_{m_{i}+1}\setminus F_{m_i}$ and
			$$\bigcap_{i=1}^n g_i^{-1}W_{s(i)} \neq
			\emptyset \mbox{\ \ for all\ } s \in \{1, 2, \ldots, l\}^n.$$
			Let $m_{k+1}$ be large enough, so that $g_1,\ldots,g_k
			\in F_{m_{k+1}}$. By Theorem \ref{thm:4-1} (2), we have
			$$\bigcap_{s \in \{1, 2, \ldots, l\}^{n}} \bigcap_{j=1}^l
			N\left(\bigcap_{ i=1}^ng_i^{-1}W_{s(i)}, W_j\right)$$ is an
			infinite subset of $G$, and hence there exists $g_{k+1} \notin
			F_{m_{k+1}}$ with $g_{k+1}\neq g_i$ for each $i=1,2,\ldots,k$ such that
			$$g_{k+1} \in   \bigcap_{s \in \{1, 2, \ldots, l\}^{n}} \bigcap_{j=1}^l
			N\left(\bigcap_{ i=1}^ng_i^{-1}W_{s(i)}, W_j\right).$$ Through this iterative process,
			we obtain the sequence $S=\{g_n\}_{n=1}^\infty$ and an increasing sequence $\{m_n\}_{n=1}^\infty$
			of positive integers satisfying the claim.
		\end{proof}
		
		By the above claim, we have
		
		\begin{align*}
			h_{\mathrm{top}}^{S, \mathbf{F}}(G, \mathcal{U}) =& \limsup_{n
				\rightarrow \infty} \frac{1}{|F_n \cap S|} \log \mathcal{N}\left(\bigvee_{g
				\in F_n \cap S} g^{-1}\mathcal{U}\right)\\
			\geqslant& \limsup_{n \rightarrow \infty}\frac{1}{|F_{m_{n+1}} \cap
				S|} \log \mathcal{N}\left(\bigvee_{g \in F_{m_{n+1}} \cap S}
			g^{-1}\mathcal{U}\right)\\
			=& \limsup_{n \rightarrow \infty}\frac{1}{|C_n|} \log
			\mathcal{N}\left(\bigvee_{g \in C_n} g^{-1}\mathcal{U}\right)\\
			=& \limsup_{n \rightarrow \infty}\frac{1}{n} \log l^n=\log l=\log
			\mathcal{N}(\mathcal{U}).
		\end{align*}
		Meanwhile, as $h_{\mathrm{top}}^{S, \mathbf{F}}(G, \mathcal{U}) \leqslant
		\log \mathcal{N}(\mathcal{U})$, we have $h_{\mathrm{top}}^{S, \mathbf{F}}(G,
		\mathcal{U})=\log \mathcal{N}(\mathcal{U})$.
		
		\medskip
		
		(1) $\Rightarrow$ (3).
		Let $\mathcal{V}=\{V_1, V_2, \ldots, V_k\}$ be a non-trivial
		finite open cover. Take $x_j \in \mathrm{int}(X \setminus V_j)$ for
		every $1 \leq j \leq k$. Set $\{y_1, y_2, \ldots,
		y_l\}=\{x_1, x_2, \ldots, x_k\}$ with $y_s \neq y_t$ for $1
		\leq s<t \leq l$. Clearly, $l \geq 2$ and we can take
		pairwise disjoint closed neighborhood $W_i$ of $y_i$, $i=1, 2,
		\ldots, l$ such that the open cover $\mathcal{U}=\{X \setminus W_1,
		X \setminus W_2, \ldots, X \setminus W_l\}$ is coarser than
		$\mathcal{V}$. Given any finite subset $F$ of $G$, let $\mathcal{P}$ be the subcover
		of $\bigvee_{g\in F} g^{-1}\mathcal{U}$ with the minimum cardinality.
		Thus, for any $s\in\{1,2,\ldots,n\}^{|F|}$, if $$\bigcap_{g\in F}g^{-1}W_{s(g)}\neq \emptyset,$$ then there exist $V\in\mathcal{P}$ with the form
		\[V=\bigcap_{g\in F}X\setminus W_{s'(g)}\text{ for some }s'\in\{1,2,\ldots,n\}^{|F|}\text{ with }s'(g)\neq s(g)\text{ for }\forall g\in F\]
		such that $V$ covers $\bigcap_{g\in F}g^{-1}W_{s(g)}$.
		Note that such $V$ cannot cover $\bigcap_{g\in F}g^{-1}W_{s(g)}\neq\emptyset$ with $s(g)=s'(g)$ for some $g\in F$.
		Thus, each such $V \in \mathcal{P}$ covers at most $(l-1)^{|F|}$ $s\in \{1,2,\ldots,n\}^{|F|}$,
		which together with the fact that $\mathcal{P}$ is an open cover of $X$ implies that
		\begin{align*}
			\mathcal{N}\left(\bigvee_{g\in F} g^{-1}\mathcal{U}\right)\cdot(l-1)^{|F|}&=\mathcal{N}\left(\mathcal{P}\right)\cdot(l-1)^{|F|}\\
			&\ge    \left|\left\{s \in \{1, 2, \ldots, l\}^F: \bigcap_{g \in
				F}g^{-1}W_{s(g)} \neq \emptyset\right\}\right|.
		\end{align*} By a proof of similar that of claim in (1) $\Rightarrow$ (2),
		there exists a sequence $S=\{g_n\}_{n=1}^\infty$ of distinct elements of $G$  and an increasing
		sequence $\{m_n\}_{n=1}^{\infty}$ of positive integers such that for each $n \geq 1$,
		$g_n\in F_{m_{n}+1}\setminus F_{m_n}$ and for any $s \in \{1, 2, \ldots, l\}^{n}$,
		$$\bigcap_{i=1}^ng_i^{-1}W_{s(i)} \neq \emptyset.$$
		Therefore, we have
		\begin{eqnarray*}
			h_{\mathrm{top}}^{S, \mathbf{F}}(G, \mathcal{U}) &=& \limsup_{n
				\rightarrow \infty} \frac{1}{|F_n \cap S|} \log \mathcal{N}\left(\bigvee_{g
				\in F_n \cap S} g^{-1}\mathcal{U}\right)\\\\
			&\geq& \limsup_{n \rightarrow \infty}\frac{1}{|F_{m_{n+1}} \cap
				S|} \log \mathcal{N}\left(\bigvee_{g \in F_{m_{n+1}} \cap S}
			g^{-1}\mathcal{U}\right)\\\\
			&=& \limsup_{n \rightarrow \infty}\frac{1}{n} \log
			\mathcal{N}\left(\bigvee_{i=1}^n g_i^{-1}\mathcal{U}\right)\\\\
			&\geq& \limsup_{n \rightarrow \infty}\frac{1}{n}\log
			\frac{\left|\displaystyle{\left\{s \in \{1, 2, \ldots, l\}^{n}:
					\bigcap_{i=1}^ng_i^{-1}W_{s(i)} \neq
					\emptyset\right\}}\right|}{(l-1)^n}\\\\
			&=& \log\frac{l}{l-1}>0.
		\end{eqnarray*}
		So $h_{\mathrm{top}}^{S, \mathbf{F}}(G, \mathcal{V}) \geqslant
		h_{\mathrm{top}}^{S, \mathbf{F}}(G, \mathcal{U})>0$.
	\end{proof}
	
	\subsection{Strong mixing}
	In this subsection, we discuss the relation between strong topological mixing
	and sequence entropy. Let us recall that a $G$-system $(X, G)$ is
	called {\it strongly (topologically) mixing} if for every two
	nonempty open subsets $U$ and $V$ of $X$, $N(U, V)$ is cofinite,
	i.e., $\{g \in G: U \cap g^{-1}V=\emptyset\}$ is finite. This definition is classical, see \cite{CKN,KR}. The classical strongly mixing example is the topological
	Bernoulli system.
	
	\medskip
	
	We have the following result.
	\begin{thm} \label{thm:4-3}
		Let $(X, G)$ be a $G$-system. If $(X, G)$ is strongly mixing, then the following
		properties hold:

		\begin{enumerate}
			\item for each admissible open cover $\mathcal{U}$ and any infinite subset $H$ of $G$, there exists an
			infinite subset $S \subset H$ such that $h_{\mathrm{top}}^{S,
				\mathbf{F}}(G, \mathcal{U})=\log \mathcal{N}(\mathcal{U})$;
			
			\medskip
			
			\item for each non-trivial finite open cover $\mathcal{U}$ and any infinite subset $H$ of $G$,
			there exists an infinite subset $S \subset H$ such that
			$h_{\mathrm{top}}^{S, \mathbf{F}}(G, \mathcal{U})>0$.
		\end{enumerate}
	\end{thm}
	
	\begin{proof}
		(1). Assume that $(X, G)$ is strongly mixing
		and let $\mathcal{U}=\{U_1, U_2, \ldots, U_l\}$ be an admissible open
		cover. Let $$W_i=\mathrm{int}\left(U_i \setminus \bigcup_{j \neq
			i}U_j\right) \mbox{\ \ for each\ } i=1, 2, \ldots, l.$$ Then $W_1,
		W_2, \ldots, W_l$ are pairwise disjoint nonempty open sets of $X$.
		
		\medskip
		
		Now we can prove the following claim.
		
		{\it Claim}. There exists a sequence $S=\{g_n\}_{n=1}^\infty$ of distinct elements of $H$  and an increasing
		sequence $\{m_n\}_{n=1}^{\infty}$ of positive integers such that for each $n \geq 1$,
		$g_n\in F_{m_{n}+1}\setminus F_{m_n}$ and, for any $s \in \{1, 2, \ldots, l\}^{n}$,
		$$\bigcap_{i=1}^ng_i^{-1}W_{s(i)} \neq \emptyset.$$
		
		\begin{proof}[Proof of Claim]
			We use induction on $n$. It is obvious that
			the claim holds for $n=1$. Assume that the claim holds for $n=k$,
			next we want to show that the claim also holds for $n=k+1$. By the
			assumption, there exist distinct elements $g_1, g_2, \ldots,
			g_k$ of $H$ and positive integers $m_1<m_2<\ldots<m_k$ such that  for each $i\in\{1,2,\ldots,k\}$,
			$g_i\in F_{m_{i}+1}\setminus F_{m_i}$ and
			$$\bigcap_{i=1}^n g_i^{-1}W_{s(i)} \neq
			\emptyset \mbox{\ \ for all\ } s \in \{1, 2, \ldots,
			l\}^n.$$
			Let  $m_{k+1}$ be large enough, so that $g_1,\ldots,g_k
			\in F_{m_{k+1}}$. Since $N(U,V)$ is cofinite for any nonempty open sets $U,V$ of $X$, it follows that
			$$\bigcap_{s \in \{1, 2, \ldots, l\}^{n}} \bigcap_{j=1}^l
			N\left(\bigcap_{ i=1}^ng_i^{-1}W_{s(i)}, W_j\right)$$ is also cofinite, and hence
			$$\bigcap_{s \in \{1, 2, \ldots, l\}^{n}} \bigcap_{j=1}^l
			N\left(\bigcap_{ i=1}^ng_i^{-1}W_{s(i)}, W_j\right)\cap H$$ is
			an infinite subset of $G$. Thus, there exists $g_{k+1} \in
			H\setminus F_{m_{k+1}}$ with $g_{k+1}\neq g_i$ for each $i=1,2,\ldots,k$ such that
			$$g_{k+1} \in   \bigcap_{s \in \{1, 2, \ldots, l\}^{n}} \bigcap_{j=1}^l
			N\left(\bigcap_{ i=1}^ng_i^{-1}W_{s(i)}, W_j\right).$$ Through this iterative process,
			we obtain the sequence $S=\{g_n\}_{n=1}^\infty$ and an increasing sequence
			$\{m_n\}_{n=1}^\infty$ of positive integers satisfying the claim.
		\end{proof}
		By the above claim, we have
		\begin{eqnarray*}
			h_{\mathrm{top}}^{S, \mathbf{F}}(G, \mathcal{U}) &=& \limsup_{n
				\rightarrow \infty} \frac{1}{|F_n \cap S|} \log \mathcal{N}\left(\bigvee_{g
				\in F_n \cap S} g^{-1}\mathcal{U}\right)\\\\
			&\geqslant& \limsup_{n \rightarrow \infty}\frac{1}{|F_{m_{n+1}} \cap
				S|} \log \mathcal{N}\left(\bigvee_{g \in F_{m_{n+1}} \cap S}
			g^{-1}\mathcal{U}\right)\\\\
			&=& \limsup_{n \rightarrow \infty}\frac{1}{|C_n|} \log
			\mathcal{N}\left(\bigvee_{g \in C_n} g^{-1}\mathcal{U}\right)\\\\
			&=& \limsup_{n \rightarrow \infty}\frac{1}{n} \log l^n=\log l=\log
			\mathcal{N}(\mathcal{U}).
		\end{eqnarray*}
		Since $h_{\mathrm{top}}^{S, \mathbf{F}}(G, \mathcal{U}) \leqslant
		\log \mathcal{N}(\mathcal{U})$, we have $h_{\mathrm{top}}^{S, \mathbf{F}}(G,
		\mathcal{U})=\log \mathcal{N}(\mathcal{U})$.
		
		\medskip
		
		Statement (2) is proved by the same argument as that of (1) $\Rightarrow$ (3) in Theorem \ref{thm:4-2}.
	\end{proof}
	\begin{rem} \label{rem:4-4}
		The converse of Theorem \ref{thm:4-3} is not true even if
		$G=\mathbb{Z}$ (see for example \cite{Bla92,HY04})
	\end{rem}
	
	\subsection{Mild mixing}In this section, we investigate mild mixing via sequence entropy. Let us begin with some notations and definitions.
	Given a $G$-system $(X,G)$, a point $x\in X$ is called
	
	\begin{enumerate}
		\item a\emph{ transitive point} of $(X,G)$ if $\overline{\{gx:g\in G\}}=X$,
		and denote by $\mathrm{Trans}(X,G)$ the set of all transitive points;
		
		\medskip
		
		\item a \emph{recurrent point} of $(X,G)$ if for any neighborhood $U$ of $x$,
		$\{g\in G:gx\in U\}$ is an infinite set, and denote by $\mathrm{Rec}(X,G)$ the set of all recurrent points.
	\end{enumerate}
	A $G$-system $(X,G)$ is said to be {\it transitive recurrent} if $\mathrm{Trans}(X,G)\cap \mathrm{Rec}(X,G)\neq \emptyset$.
	
	\begin{prop}\label{prop:4-5}
		A $G$-system $(X, G)$ is transitive recurrent if and only if it
		is infinite transitive, i.e., $N(U, V)$ is infinite for every
		two nonempty open subsets $U$ and $V$ of $X$.
	\end{prop}
	
	\begin{proof}
		Assume that $(X, G)$ is transitive recurrent and $x \in \mathrm{Trans}(X,G)\cap
		\mathrm{Rec}(X,G)$. Then $N(x, U)$ is infinite for every nonempty open
		subset $U$ of $X$, where $N(x, U)=\{g \in G: gx \in U\}$. Thus,
		for every pair nonempty open subsets $U$ and $V$ of $X$,
		$N(U, V)=N(x, V)N(x, U)^{-1}$ is infinite.
		
		\medskip
		
		Conversely, suppose that $(X, G)$ is infinite transitive, that is, $N(U, V)$ is infinite for every
		two nonempty open subsets $U$ and $V$ of $X$. Then $N(x, U)$ is
		infinite for every $x \in \mathrm{Trans}(X, G)$ and every
		nonempty open subsets $U$ of $X$. Thus, $\emptyset \neq \mathrm{Trans}(X, G) \subset \mathrm{Rec}(X,
		G)$, which implies that $(X, G)$ is transitive recurrent.
	\end{proof}
	
	\medskip
	
	A $G$-system $(X, G)$ is called {\it mildly mixing} if $(X \times Y, G)$ is transitive recurrent for any
	transitive recurrent system $(Y, G)$. From Proposition \ref{prop:4-5} and the definition, it is easy to check that
	\begin{equation}\label{eq:mild to weak}
		\mbox{strong mixing} \Rightarrow \mbox{mild mixing} \Rightarrow \mbox{weak mixing}.
	\end{equation}
	
	\begin{rem}\label{rem:review} (1) Recall that for a $\mathbb{Z}_+$-system $(X,T)$, the transitivity is
		defined by for any nonempty open subsets $U$ and $V$,
		$\{n\in\mathbb{N}: U\cap T^{-n}V\neq \emptyset\}$ is nonempty. It
		well known that if $(X, T)$ is transitive, then $\mathrm{Trans}(X,
		T)$ is a dense $G_{\delta}$-set of $X$, and $\mathrm{Trans}(X, T)
		\subset \mathrm{Rec}(X, T)$. Clearly, if $\mathrm{Trans}(X,T)\cap \mathrm{Rec}(X,T) \neq \emptyset$
		then $(X,T)$ is transitive. Thus, for $\mathbb{Z}_+$-actions,
		transitivity is equivalent to transitive recurrent.
		
		\medskip
		
		(2) For a $\mathbb{Z}_+$-system $(X, T)$, it is called mild mixing
		if $(X \times Y, T \times S)$ is transitive for any transitive
		system $(Y, S)$ (see \cite{Gla04} or \cite{HY04}). However, for
		group actions, the reviewer pointed out to us the following claim:
		
		\medskip
		
		{\it Claim}.\ The only topologically mildly mixing system (according
		to the definition of mild mixing for $\mathbb{Z}_+$-actions in
		\cite{Gla04} or \cite{HY04}) is the trivial system.
		
		\begin{proof}[Proof of the Claim.]
			Let $G$ be any nontrivial countable group and let $X$ be the
			one-point compactification of $G$ viewed as a discrete set (if $G$
			is finite then we let $X=G$). If we agree that $g \cdot
			\infty=\infty$ for any $g \in G$, then $G$ acts on $X$ by
			homeomorphisms, via left multiplication. Clearly, $(X, G)$ is
			transitive, with any $x \neq \infty$\ (i.e., $x \in G$) being a
			transitive point.
			
			\medskip
			
			Now, let $(Y, G)$ be any $G$-action, where $Y$ has at least two
			different points. Suppose that $(Y, G)$ is topologically mildly
			mixing. Then the product $X \times Y$ should be transitive under the
			product action of $G$. Let $\left(x_0, y_0\right)$ be a transitive
			point in $X \times Y$. Clearly, we have $x_0 \neq \infty$, i.e.,
			$x_0=g_0 \in G$ and there exists $y_1 \neq y_0$ in $Y$. By
			transitivity, exists a sequence $\left\{g_n\right\}_{n \geq 1}$ in
			$G$ such that $g_n\left(x_0, y_0\right)=\left(g_n g_0,
			g_ny_0\right)$ converges to $(g_0, y_1)$. In particular, $g_n g_0$
			converges to $g_0$. Since $G$ is discrete, it must be that $g_n
			=e_G$ (the unit of $G$) for large enough $n$. But then
			$g_n\left(x_0, y_0\right)=\left(x_0, y_0\right)$ which does not
			converge to $\left(x_0, y_1\right)$, and we have a contradiction.
		\end{proof}
	\end{rem}

	Denote by $\mathcal{P}_f(\mathbb{N})$ the set of all finite nonempty subsets
	of $\mathbb{N}$. Let $G$ be a group and $\{p_i\}_{i=1}^{\infty}$ be a
	sequence in $G$. Define
	$$FP\left(\{p_i\}_{i=1}^{\infty}\right)=\left\{\prod_{i \in F}p_i: F \in
	\mathcal{P}_f(\mathbb{N})\right\}, \mbox{\ where\ }\prod_{i \in
		F}p_i=p_{n_1} \cdot p_{n_2} \cdot \ldots \cdot p_{n_k}$$ for
	$F=\{n_1, n_2, \ldots, n_k\} \in \mathcal{P}_f(\mathbb{N})$ with
	$n_1<n_2<\cdots<n_k$.
	For each $n \in \mathbb{N}$, the {\it initial $n$-segment} of
	$FP(\{p_i\}_{i=1}^{\infty})$ is defined as
	$$FP(\{p_i\}_{i=1}^n)=\left\{\prod_{i \in F}p_i: F \in
	\mathcal{P}_f(\{1, 2, \ldots, n\})\right\}.$$ A subset $S$ of $G$ is
	called an {\it infinite $IP$-set} if there exists a sequence
	$\{p_i\}_{i=1}^{\infty}$ in $G$ such that
	$FP(\{p_i\}_{i=1}^{\infty})$ is infinite and
	$FP(\{p_i\}_{i=1}^{\infty}) \subseteq S$.
	
	\medskip
	
	To study the mild mixing by sequence entropy, we need the following version of Furstenberg's realization theorem of $IP$-sets,
	which was generalized from the additive semigroup of positive integers (see \cite[Theorem 2.17]{Fur81}) to discrete groups (see \cite[Theorem 2.8]{DLX}).
	\begin{lem}\label{lem:Furstenberg's realization theorem of IP-sets}
		Let $G$ be an infinite discrete group and $A$ be an infinite IP-set of $G$. Then there exists a $G$-system $(X,G)$, $x\in \operatorname{Rec}(X,G)$, and a neighborhood $U$ of $x$ such that
		\[\{g\in G:gx\in U,g\neq e\}\subset A.\]
	\end{lem}
	
	With the help of Lemma \ref{lem:Furstenberg's realization theorem of IP-sets}, we have the following result
	to describe the set $N(U,V)$ for any nonempty open sets $U,V$ of $X$.
	
	\begin{lem} \label{lem:4-6}
		Let $(X, G)$ be a $G$-system. If $(X, G)$ is mildly mixing, then
		$N(U, V) \cap AA^{-1}$ is infinite for every pair of nonempty open
		subsets $U, V$ of $X$ and every infinite $IP$-set $A$ of $G$.
	\end{lem}
	\begin{proof}
		Fix a pair of nonempty open
		subsets $U, V$ of $X$ and an infinite $IP$-set $A$ of $G$.
		By Lemma \ref{lem:Furstenberg's realization theorem of IP-sets}, there exists a $G$-system
		$(Y,G)$, $y\in \operatorname{Rec}(Y,G)$, and an open neighborhood $W$ of $y$ such that
		\[\{g\in G:gy\in W,g\neq e\}\subset A.\]
		Let $Z=\overline{\{gy:g\in G\}}$ be a $G$-invariant closed subset of $Y$. Then
		\begin{equation}\label{eq2}
			y\in \operatorname{Rec}(Z,G)\cap \operatorname{Trans}(Z,G)
		\end{equation}
		and
		\begin{equation}\label{eq3}
			\{g\in G: gy\in W\cap Z,g\neq e\}=\{g\in G: gy\in W,g\neq e\}\subset A.
		\end{equation}
		By \eqref{eq2}, $(Z,G)$ is recurrent transitive, and hence $(X\times Z,G)$ is transitive recurrent, as $(X,G)$ is mildly mixing. Thus,
		\begin{equation}\label{eq4}
			N(U,V)\cap N(W,W)=N(U\times W, V\times W) \mbox{\ is infinite}.
		\end{equation}
		Now we prove that $N(W,W)\subset AA^{-1}$. Indeed, for any $g\in N(W,W)$, $W\cap g^{-1}W$ is a nonempty open subset.
		By \eqref{eq2}, there exists $h\in G$ with $h \neq e_G, g^{-1}$ such that $hy\in W\cap g^{-1} W$,
		that is, $gy, ghy\in W$. Since $h,gh\neq e_G$, it follows from \eqref{eq3} that $h,gh\in A$,
		which implies that $g=(gh)h^{-1}\in AA^{-1}$. By \eqref{eq4}, we have that $N(U,V)\cap AA^{-1}$ is infinite.
		This finishes the proof of the lemma.
	\end{proof}
	
	With the help of Lemma \ref{lem:4-6}, we obtain the following description of topologically mildly mixing systems.
	
	\begin{thm} \label{thm:4-7}
		Let $(X,G)$ be a $G$-system. If $(X, G)$ is mildly mixing, then the following properties
		hold:
		
		\begin{enumerate}
			\item for each admissible open cover $\mathcal{U}$ and each
			infinite $IP$-set $H$, there exists an infinite subset $S \subset H$
			such that $h_{\mathrm{top}}^{S, \mathbf{F}}(G, \mathcal{U})=\log
			N(\mathcal{U})$;

			\item for each non-trivial finite open cover $\mathcal{U}$ and each
			infinite $IP$-set $H$, there exists an infinite subset $S \subset H$
			such that $h_{\mathrm{top}}^{S, \mathbf{F}}(G, \mathcal{U})>0$.
		\end{enumerate}
	\end{thm}
	
	\begin{proof}
		Assume that $(X, G)$ is mildly mixing, let $\mathcal{U}=\{U_1, U_2,
		\ldots, U_l\}$ be an admissible open cover and let
		$H=FP(\{p_i\}_{i=1}^{\infty})$ be an $IP$-set. Let
		$$W_i=\mathrm{int}\left(U_i \setminus \bigcup_{j \neq i}U_j\right)
		\mbox{\ \ for each\ } i=1, 2, \ldots, l.$$ Then $W_1, W_2, \ldots,
		W_l$ are pairwise disjoint nonempty open sets of $X$.
		
		\medskip
		
		{\it Claim}. Fix a finite subset $F$ of $G$. For any $n \in
		\mathbb{N}$, there exists a subset $C_n=\{g_{1,n}, g_{2,n}, \ldots,
		g_{n,n}\} \subset H$ such that $C_n \cap F=\emptyset$ and
		$$\bigcap_{g \in C_n}g^{-1}W_{s(g)} \neq \emptyset \mbox{\ \ for any\ } s \in \{1, 2, \ldots, l\}^{C_n}.$$
		
		\begin{proof}[Proof of Claim]
			We use induction on $n$. It is obvious that
			the claim holds for $n=1$. Assume that the claim holds for $n=k$.
			Next we want to show that the claim also holds for $n=k+1$. By the
			assumption, there exists a finite subset $C_k=\{g_{1, k}, g_{2, k},
			\ldots, g_{k,k}\} \subset H$ such that $C_k \cap F=\emptyset$ and
			$$\bigcap_{g \in C_k}g^{-1}W_{s(g)} \neq \emptyset \mbox{\ \ for any\ } s \in \{1, 2, \ldots, l\}^{C_k}.$$
			
			Take $m_k \in \mathbb{N}$ large enough such that $C_k \subset
			FP(\{p_i\}_{i=1}^{m_k})$. Let
			$H_{m_k}=FP(\{p_i\}_{i=m_k+1}^{\infty})$.
			From \eqref{eq:mild to weak}, the system $(X,G)$ is weakly mixing, which together with Theorem \ref{thm:4-1}, implies that there exist two nonempty open sets $U,V$ with
			\[N(U,V)\subset \bigcap_{s \in \{1, 2, \ldots, l\}^{C_k}} \bigcap_{j=1}^l
			N\left(\bigcap_{g \in C_k}g^{-1}W_{s(g)}, W_j\right).\]
			Thus, by Lemma \ref{lem:4-6}, we
			have
			$$\left[\bigcap_{s \in \{1, 2, \ldots, l\}^{C_k}} \bigcap_{j=1}^l
			N\left(\bigcap_{g \in C_k}g^{-1}W_{s(g)}, W_j\right)\right] \cap
			H_{m_k}H_{m_k}^{-1}$$ is infinite, and then there exist $g_{1},
			g_{2} \in H_{m_{k}}$ such that
			$$g_{1}g_{2}^{-1} \in \bigcap_{s \in\{1,2, \ldots, l\}^{C_k}} \bigcap_{j=1}^{l} N\left(\bigcap_{i=1}^{k}
			(g_{i,k})^{-1} W_{s(g_{i,k})}, W_{j}\right), g_1g_2^{-1} \notin C_k
			\mbox{ and } g_{1},g_{i,k}g_2\notin F.$$ Let $g_{i,
				k+1}=g_{i,k}g_2$, $i=1, 2, \ldots, k$ and $g_{k+1, k+1}=g_1$. Then
			$$C_{k+1}=\{g_{1, k+1}, g_{2, k+1}, \ldots, g_{k+1, k+1}\} \subset H,$$
			$C_{k+1} \cap F =\emptyset$ and for any $s \in \{1, 2, \ldots,
			l\}^{C_k}$ and every $j=1, 2, \ldots, l$,
			$$\bigcap_{g \in C_k}g^{-1}W_{s(g)} \cap g_2g_1^{-1}W_j \neq \emptyset,$$
			and so
			$$\bigcap_{g \in C_{k+1}} g^{-1}W_{s(g)} \neq
			\emptyset \mbox{\ \ for all\ } s \in \{1, 2, \ldots,
			l\}^{C_{k+1}}.$$ This shows the claim holds for $n=k+1$ and so
			completes the proof of the claim.
		\end{proof}
		
		By claim, we can choose an increasing sequence
		$\{m_i\}_{i=1}^{\infty}$ of positive integers such that $C_{n_i}
		\subset F_{m_{i+1}}$, $C_{n_i} \cap F_{m_i}=\emptyset$ and
		$|C_{n_i}|=n_i=(i+1)!-i!$. Let $S=\bigcup_{i=1}^{\infty}C_{n_i}
		\subset H$ and $S_i=\bigcup_{j=1}^i C_{n_j}$. Then we have
		\begin{eqnarray*}
			h_{\mathrm{top}}^{S, \mathbf{F}}(G, \mathcal{U}) &=& \limsup_{n
				\rightarrow \infty} \frac{1}{|F_n \cap S|} \log \mathcal{N}\left(\bigvee_{g
				\in F_n \cap S} g^{-1}\mathcal{U}\right)\\
			&\geqslant& \limsup_{i \rightarrow \infty}\frac{1}{|F_{m_{i+1}} \cap
				S|} \log \mathcal{N}\left(\bigvee_{g \in F_{m_{i+1}} \cap S}
			g^{-1}\mathcal{U}\right)\\\\
			&=& \limsup_{i \rightarrow \infty}\frac{1}{|S_i|} \log
			\mathcal{N}\left(\bigvee_{g \in S_i} g^{-1}\mathcal{U}\right)\\
			&\geqslant& \limsup_{i \rightarrow
				\infty}\frac{1}{\displaystyle{\sum_{j=1}^i n_j}} \log
			\mathcal{N}\left(\bigvee_{g \in C_{n_i}} g^{-1}\mathcal{U}\right)\\
			&=& \limsup_{i \rightarrow \infty}\frac{\log
				l^{n_i}}{\displaystyle{\sum_{j=1}^i n_j}}=\limsup_{i \rightarrow
				\infty}\frac{n_i}{\displaystyle{\sum_{j=1}^i n_j}}\log l=\log l=\log
			\mathcal{N}(\mathcal{U}).
		\end{eqnarray*}
		Since $h_{\mathrm{top}}^{S, \mathbf{F}}(G, \mathcal{U}) \leqslant
		\log \mathcal{N}(\mathcal{U})$, we have $h_{\mathrm{top}}^{S, \mathbf{F}}(G,
		\mathcal{U})=\log \mathcal{N}(\mathcal{U})$. This finishes the proof of (1).
		
		\medskip
		
		Now we prove (2). Let $\mathcal{V}=\{V_1, V_2, \ldots, V_k\}$ be a non-trivial
		finite open cover. Take $x_j \in \mathrm{int}(X \setminus V_j)$ for
		every $1 \leqslant j \leqslant k$. Set $\{y_1, y_2, \ldots,
		y_l\}=\{x_1, x_2, \ldots, x_k\}$ with $y_s \neq y_t$ for $1
		\leqslant s<t \leqslant l$. Clearly, $l \geqslant 2$ and we can take
		pairwise disjoint closed neighborhood $W_i$ of $y_i$, $i=1, 2,
		\ldots, l$ such that the open cover $\mathcal{U}=\{X \setminus W_1,
		X \setminus W_2, \ldots, X \setminus W_l\}$ is coarser than
		$\mathcal{V}$. By a proof similar to that of inequalities of the same type in Theorem 4.2, we can obtain that
		$$(l-1)^{|F|}\mathcal{N }\left(\bigvee_{g \in F} g^{-1}\mathcal{U}\right) \geqslant
		\left|\left\{s \in \{1, 2, \ldots, l\}^F: \bigcap_{g \in
			F}g^{-1}W_{s(g)} \neq \emptyset\right\}\right|$$ for every finite
		subset $F$ of $G$. Moreover, by a proof similar to that of (1), we can find an increasing sequence $\{m_i\}_{i=1}^\infty$, an infinite subset
		$S=\bigcup_{i=1}^{\infty}C_{n_i} \subset H$ and $S_i=\bigcup_{j=1}^i
		C_{n_j}$ such that $C_{n_i}\subset F_{m_{i+1}}$, $C_{n_i}\cap F_{m_i}=\emptyset$, $|C_{n_i}|=(i+1)!-i!$
		and $$\bigcap_{g\in C_{n_i}}g^{-1}W_{s(g)}\neq\emptyset \mbox{\ for any\ } s\in\{1,2,\ldots,l\}^{C_{n_i}}.$$
		Thus,
		\begin{align*}
			h_{\mathrm{top}}^{S, \mathbf{F}}(G, \mathcal{U}) &= \limsup_{n
				\rightarrow \infty} \frac{1}{|F_n \cap S|} \log \mathcal{N}\left(\bigvee_{g
				\in F_n \cap S} g^{-1}\mathcal{U}\right)\\
			&\geqslant \limsup_{i \rightarrow \infty}\frac{1}{|F_{m_{i+1}} \cap
				S|} \log \mathcal{N}\left(\bigvee_{g \in F_{m_{i+1}} \cap S}
			g^{-1}\mathcal{U}\right)\\
			&= \limsup_{i \rightarrow \infty}\frac{1}{|S_i|} \log
			\mathcal{N}\left(\bigvee_{g \in S_i} g^{-1}\mathcal{U}\right)\\
			&\geqslant \limsup_{i \rightarrow
				\infty}\frac{1}{\displaystyle{\sum_{j=1}^i n_j}} \log
			\mathcal{N}\left(\bigvee_{g \in C_{n_i}} g^{-1}\mathcal{U}\right)\\
			&\geqslant \limsup_{i \rightarrow \infty}\frac{\displaystyle{\log
					\frac{1}{(l-1)^{n_i}}\left|\displaystyle{\left\{s \in \{1, 2,
						\ldots, l\}^{C_{n_i}}: \bigcap_{g \in C_{n_i}}g^{-1}W_{s(g)} \neq
						\emptyset\right\}}\right|}}{\displaystyle{\sum_{j=1}^i n_j}}\\
			&= \limsup_{i \rightarrow
				\infty}\frac{\displaystyle{\log\frac{1}{(l-1)^{n_i}}
					l^{n_i}}}{\displaystyle{\sum_{j=1}^i n_j}}=\lim_{i \rightarrow
				\infty}\frac{n_i}{\displaystyle{\sum_{j=1}^i
					n_j}}\log\frac{l}{l-1}=\log\frac{l}{l-1}>0.
		\end{align*}
		Hence, $h_{\mathrm{top}}^{S, \mathbf{F}}(G, \mathcal{V}) \geqslant
		h_{\mathrm{top}}^{S, \mathbf{F}}(G, \mathcal{U})>0$.
	\end{proof}

	\section{Sequence entropy pairs and null systems}
	
	In this section, we  use sequence entropy pairs to describe null systems and weakly mixing
	systems.
	
	\begin{de} \label{de:5-1}
		Let $(X,G)$ be a $G$-system.
		
		\begin{enumerate}
			\item We say that $\left(x_{1}, x_{2}\right) \in X \times X$ is a {\it sequence entropy pair} if $x_{1} \neq x_{2}$
			and if whenever $U_{i}$ are closed mutually disjoint neighborhoods
			of points $x_{i}, i=1,2$, there exists an infinite subset $S
			\subset G$ such that
			$h_{\mathrm{top}}^{S,\textbf{F}}\left(G,\left\{X\setminus U_{1},
			X\setminus U_{2}\right\}\right)>0$.

			\item We say that $(X, G)$ has a {\it uniform positive sequence entropy}
			(for short u.p.s.e.), if every point $\left(x_{1}, x_{2}\right) \in
			X \times X$, not in the diagonal $\Delta=\{(x, x): x \in X\}$, is a
			sequence entropy pair.

			\item We say that $(X, G)$ is a {\it null system}, if for any infinite subset $S \subset G$ such that $h_{\mathrm{top}}^{S,
				\mathbf{F}}(G)=0$.
		\end{enumerate}
	\end{de}
	
	\begin{rem} \label{rem:5-2}
		It is easy to see that $(X, G)$ has u.p.s.e. if and only if for any
		cover $\mathcal{U}=\{U, V\}$ of $X$ by two non-dense open sets, one
		has $h_{top}^{S,\textbf{F}}(G, \mathcal{U})>0$ for some infinite
		subset $S\subset G$.
	\end{rem}
	
	Denote by $SE(X,G)$ the set of all sequence entropy pairs.
	Following ideas in \cite{Bla93}, we provide a sufficient condition for the existence of sequence entropy pairs.
	\begin{lem}\label{lem:5-4}
		If there is a standard open cover $\mathcal{U}=\{U, V\}$ of $X$ with
		$h_{\mathrm{top}}^{S,\textbf{F}}(G, \mathcal{U})>0$ for some
		infinite sequence $S \subset G$, then there are points $x \in X\setminus U$
		and $x' \in X\setminus V$ such that $(x, x')$ is a sequence entropy pair.
	\end{lem}
	\begin{proof}
		Firstly, we show that one can find a strictly coarser cover $\mathcal{U}_1=\{U_1,V_1\}$ with $h_{\mathrm{top}}^{S,\textbf{F}}(G, \mathcal{U})>0$,
		having the property that $\operatorname{diam}(X\setminus U_1)<1/2\operatorname{diam}(X\setminus U)$
		and $\operatorname{diam}(X\setminus V_1)<1/2\operatorname{diam}(X\setminus V)$,
		where $\operatorname{diam}(A):=\max\{d(x,y):x,y\in A\}$ for $A\subset X$.
		Note that $X\setminus U$ cannot be a singleton, because in this case,
		$h_{\mathrm{top}}^{S,\textbf{F}}(G, \mathcal{U})=0$: if $X\setminus U$ is a singleton,
		then $\cap_{g\in F_n\cap S}g^{-1}U$ is equal to $X$ minus at most $n$ points,
		and hence $\bigvee_{g \in F_n \cap S}g^{-1}\mathcal{U}$ has a subcover with cardinality at most $n+1$,
		which implies that $h_{\mathrm{top}}^{S,\textbf{F}}(G, \mathcal{U})=0$.
		Thus, there exist two distinct points $y,y'\in X\setminus U$. Fix $\epsilon_1$ such that $0<\epsilon_1<\frac{1}{2}d(y,y')$,
		and construct a finite cover of $X\setminus U$ by open balls with radius $\epsilon_1$ centered in $X\setminus U$;
		call is $\mathcal{V}=\{U_1',\ldots,U'_k\}$, where $k\ge 2$. Let $F_i'=\overline{U_i'}$
		and $F_i=F_i'\cap U$. $i=1,2,\ldots,k$. Consider the open cover $\bigvee_{i=1}^k\{X\setminus F_i,V\}$.
		Since $\bigcap_{i=1}^k(X\setminus F_i)\subset U$, it follows that $\mathcal{U}$ is coarser than $\bigvee_{i=1}^k\{X\setminus F_i,V\}$. Thus,
		\[0<h_{\mathrm{top}}^{S,\textbf{F}}(G, \mathcal{U})\le h_{\mathrm{top}}^{S,\textbf{F}}(G, \bigvee_{i=1}^k\{X\setminus F_i,V\})\le\sum_{k=1}^nh_{\mathrm{top}}^{S,\textbf{F}}(G, \{X\setminus F_i,V\}),\]
		which implies that there exists $j\in\{1,2,\ldots,k\}$ such that
		$h_{\mathrm{top}}^{S,\textbf{F}}(G, \{X\setminus F_j,V\})>0$.
		Denote $U_1=X\setminus F_j$. Then choosing a suitable $\epsilon'$ and doing the same for $V$,
		we obtain  $V_1$ such that  $\mathcal{U}_1=\{U_1,V_1\}$ is a strictly coarser cover than $\mathcal{U}$
		with $h_{\mathrm{top}}^{S,\textbf{F}}(G, \mathcal{U})>0$, having the property that
		$\operatorname{diam}(X\setminus U_1)<\frac{1}{2}\operatorname{diam}(X\setminus U)$ and $\operatorname{diam}(X\setminus V_1)<\frac{1}{2}\operatorname{diam}(X\setminus V)$.
		
		Repeating this process infinitely many times, we obtain two decreasing sequences of closed subsets
		$\{X\setminus U_i\}_{i=1}^\infty$ and $\{X\setminus V_i\}_{i=1}^\infty$ with
		$\operatorname{diam}(X\setminus U_{i+1})<\frac{1}{2}\operatorname{diam}(X\setminus U_{i})$
		and $\operatorname{diam}(X\setminus V_{i+1})<\frac{1}{2}\operatorname{diam}(X\setminus V_{i})$
		for each $i\in\mathbb{N}$ such that $h_{\mathrm{top}}^{S,\textbf{F}}(G, \mathcal{U}_i)>0$,
		where $\mathcal{U}_i=\{U_i,V_i\}$ for each $i\in\mathbb{N}$. Since $X$ is compact, we deduce that
		\begin{equation}\label{eq5}
			\bigcap_{i=1}^\infty (X\setminus U_i)=\{x\}\text{ and }\bigcap_{i=1}^\infty (X\setminus V_i)=\{x'\}.
		\end{equation}
		
		Now we prove that $(x,x')\in SE(X,G)$. Let $W,W'$ be two closed mutually disjoint neighborhoods of $x,x'$, respectively.
		Choose $\epsilon>0$ such that closed balls $\overline{B(x,\epsilon)}\subset W$ and $\overline{B(x',\epsilon)}\subset W'$.
		Then the open cover $\{X\setminus\overline{B(x,\epsilon)},X\setminus\overline{B(x',\epsilon)}\}$
		is coarser than $\{X\setminus W,X\setminus W'\}$. By \eqref{eq5}, there exists $i\in\mathbb{N}$ such that
		$X\setminus U_i\subset \overline{B(x,\epsilon)}$ and $X\setminus V_i\subset \overline{B(x',\epsilon)}$,
		which implies that the open cover $\mathcal{U}_i$ is coarser than
		$\{X\setminus\overline{B(x,\epsilon)},X\setminus\overline{B(x',\epsilon)}\}$.
		Thus, $\mathcal{U}_i$ is coarser than $\{X\setminus W,X\setminus W'\}$, which implies that
		\[0<h_{\mathrm{top}}^{S,\textbf{F}}(G, \mathcal{U}_i)\le h_{\mathrm{top}}^{S,\textbf{F}}(G, \{X\setminus W,X\setminus W'\}).\]
		Therefore, $(x,x')\in SE(X,G)$, proving this lemma.
	\end{proof}
	With the help of Lemma \ref{lem:5-4}, we provide a characterization of null systems via sequence entropy pairs.
	\begin{thm} \label{thm:5-3}
		Let $(X,G)$ be a $G$-system. Then $(X,G)$ is null if and only if
		$SE(X,G)=\emptyset$.
	\end{thm}
	\begin{proof}
		It is easy to see that if $(X,G)$ is null then $SE(X,G)=\emptyset$.
		Conversely, if $(X,G)$ is not null then there exists a standard open cover $\mathcal{U}$ such that $h_{\mathrm{top}}^{S,\textbf{F}}(G, \mathcal{U})>0$, and hence by Lemma \ref{lem:5-4},
		there exists $(x,x')\in SE(X,G)$. This finishes the proof of Theorem \ref{thm:5-3}.
	\end{proof}
	
	Taking advantage of Theorem \ref{thm:4-2}, we immediately obtain the
	following description of weakly mixing systems via sequence entropy
	pairs.
	
	\begin{thm} \label{thm:5-5}
		Let $G$ be an Abelian group. Then the $G$-system $(X,G)$ is weakly
		mixing if and only if it is u.p.s.e.
	\end{thm}
	
	\section{Example}
	
	In this section, we provide an example to show that there exists a $\mathbb{Z}^2$-system
	such that it has zero entropy, but there exists a sequence  such that it has positive sequence entropy both in measure-theoretic and topological sense.
	\begin{ex}
		Following ideas of Example 5.4 in \cite{LX2023}, let $Y=\{0,1\}$ and $\mu(\{0\})=\mu(\{1\})=1/2$. Let $X=Y^\mathbb{Z}$ and $\nu=\mu^\mathbb{Z}$.
		Define $T_1=Id_X:X\to X$ by $$T_1(\{x_n\}_{n\in\mathbb{Z}})=\{x_n\}_{n\in\mathbb{Z}}\text{ for any }\{x_n\}_{n\in\mathbb{Z}}\in X.$$
		Define $T_2:X\to X$ by $$T_2(\{x_n\}_{n\in\mathbb{Z}})=\{x_{n+1}\}_{n\in\mathbb{Z}}
		\text{ for any }\{x_n\}_{n\in\mathbb{Z}}\in X,$$
		that is, $T_2$ is the two-sided full shift. Define a $\mathbb{Z}^2$-action $T$ by $$T^{(m,n)}=T_1^mT_2^n\text{ for all }(m,n)\in\mathbb{Z}^2.$$
		Then $(X,T)$ is a $\mathbb{Z}^2$-system, and $\mu$ induces a $\mathbb{Z}^2$-MPS. Let $F_n=[0,n-1]\times [0,n-1]$ for each $n\in\mathbb{N}$.
		Then $\textbf{F}=\{F_n\}_{n=1}^\infty$ is a F\o lner sequence.
		
		\medskip
		
		Now, we prove $(X,T)$ has zero topological entropy. Indeed, for any finite open cover $\mathcal{U}$, we have that
		\[\mathcal{N}\left(\bigvee_{i,j=0}^{n-1}T^{-(i,j)}\mathcal{U}\right)\le \mathcal{N}\left(\bigvee_{i=0}^{n-1}T_1^{-i}\mathcal{U}\right)^n\text{ for each }n\in\mathbb{N}.\]
		Since $T_1=Id$ has zero topological entropy, it follows that
		\[h_{\mathrm{top}}(T,\mathcal{U})=\lim_{n\to \infty}\frac{1}{|F_n|}\log\mathcal{N}\left(\bigvee_{i,j\in F_n}T^{-(i,j)}\mathcal{U}\right)\le \lim_{n\to \infty}\frac{1}{n}\log\mathcal{N}\left(\bigvee_{i=0}^{n-1}T_1^{-i}\mathcal{U}\right)=0.\]
		Since $\mathcal{U}$ is arbitrary, we have that $h_{\mathrm{top}}(\mathbb{Z}^2)=\sup_{\mathcal{U}}h_{\mathrm{top}}(T,\mathcal{U})=0$. However, we take $S=\{(0,n)\}_{n=1}^{\infty}$. It is known that the topological entropy of the two-sided full shift is $\log 2$. Thus,
		\begin{equation*}
			\begin{split}
				h_{\mathrm{top}}^{S,\textbf{F}}(\mathbb{Z}^2)&=\sup_{\mathcal{U}}\limsup_{n\to \infty}\frac{1}{|F_n \cap S|}\log\mathcal{N}\left(\bigvee_{i,j\in F_n \cap S}T^{-(i,j)}\mathcal{U}\right)\\
				&=\sup_{\mathcal{U}}\lim_{n\to \infty}\frac{1}{n}\log\mathcal{N}\left(\bigvee_{i=0}^{n-1}T_2^{-i}\mathcal{U}\right)=\log{2}.
			\end{split}
		\end{equation*}
		
		By an argument similar to that of the topological entropy, we also can prove that
		$h_{\nu}(\mathbb{Z}^2)=0$, but
		$h_{\nu}^{S,\textbf{F}}(\mathbb{Z}^2)=\log2$.
	\end{ex}

	\
	
	\noindent {\bf Acknowledgments.} We thank Wen Huang, Song Shao and Xiangdong Ye
	for their valuable suggestions over the topic. We would like to thank the referee for many valuable
	and constructive comments and suggestions, especially providing a reasonable definition of mild mixing
	for group actions and the claim in Remark \ref{rem:review}, which help us to improve the paper.
	

\end{document}